\documentclass[10pt, a4paper]{amsart}
\usepackage{fancyhdr}
\usepackage{stmaryrd}
\usepackage{calrsfs}

\usepackage{amsmath}
\usepackage[nospace,noadjust]{cite}
\usepackage{amsfonts}
\usepackage{amssymb,enumerate}
\usepackage{amsthm}
\usepackage{comment}
\usepackage{color}
\usepackage[all]{xy}
\usepackage{hyperref}
\usepackage{lineno}
\usepackage{tikz-cd}
\usepackage{MnSymbol}
\usetikzlibrary{decorations.pathreplacing}

\newtheorem{theorem}{Theorem}[section]
\newtheorem{prop}[theorem]{Proposition}

\newtheorem{lemma}[theorem]{Lemma}
\newtheorem{cor}[theorem]{Corollary}
\newtheorem*{main-theorem}{Main Theorem}

\theoremstyle{definition}
\newtheorem{definition}[theorem]{Definition}
\newtheorem{remark}[theorem]{Remark}

\numberwithin{equation}{section}

\DeclareMathOperator{\Frac}{Frac}
\DeclareMathOperator{\lcm}{lcm}
\DeclareMathOperator{\Mat}{Mat}
\DeclareMathOperator{\Diag}{Diag}
\DeclareMathOperator{\mult}{mult}
\DeclareMathOperator{\spec}{Spec}
\DeclareMathOperator{\mspec}{mSpec}
\DeclareMathOperator{\supp}{Supp}
\DeclareMathOperator{\msupp}{mSupp}

\hypersetup{
	pdftitle={LCM splittings},
	colorlinks=true,
	linkcolor=blue,
	citecolor=red,
	urlcolor=blue
}

\keywords{lcm splitting, gcd, B\'{e}zout domain, good factorization, adequate, elementary divisor}

\subjclass[2020]{Primary: 13A05, 13F15, 13G05; Secondary: 06F20, 13P05, 14A05}

\begin{document}
\mbox{}
\title{LCM splittings in GCD Domains}
\author{Thomas Preu}
\address{Institut f\"{u}r Mathematik\\UZH\\CH-8057 Z\"{u}rich}
\email{preu@math.uzh.ch}
\date{\today}

\begin{abstract}
In a GCD ring $R$, every pair of elements $a,b\in R$ admits a greatest common divisor $\gcd(a,b)$ and a least common
multiple $\lcm(a,b)$. Motivated by a question on Mathematics Stack Exchange by Ash Seifert~\cite{Sei22}, I introduce
the notion of an lcm splitting $c,d\in R$ for $a,b$ satisfying $cd\sim\lcm(a,b), c\mid a, d\mid b, \gcd(c,d)\sim1$. I
show that adequate B\'{e}zout domains always admit lcm splittings, give examples of B\'{e}zout domains in which some pair
has no lcm splitting, and make further remarks on this seemingly new property.
\end{abstract}
\medskip

\maketitle
{\small \tableofcontents}

\bigskip
\newpage
\section{LCM splittings in GCD rings}
\label{sec:gcdDomains}

The notion of $\gcd$ for integers is classic, giving rise to arguably the first proper known algorithm, the Euclidean
algorithm from at least 300 BC computing it. GCD rings generalize this notion to rings, which in this paper are
always associative, commutative and unital.

\begin{definition} \label{def:gcd}
Let $R$ be a ring and $a,b\in R$ two elements. 
\begin{enumerate}
\item[(a)] A \textbf{greatest common divisor} $g=\gcd(a,b)\in R$ for the pair $a,b$ is a common divisor ($g\mid a, g\mid b
  \Leftrightarrow \exists a',b'\in R: a=a'g, b=b'g$) such that any other common divisor $d\in R$ of $a$ and $b$
  necessarily divides $g$, i.e. $d\mid g$.
\item[(b)] $R$ is a \textbf{GCD ring} if every pair of elements admits a gcd.
\item[(c)] A \textbf{least common multiple} $l=\lcm(a,b)\in R$ is a common multiple ($a\mid l$ and $b\mid l$) such that any
  other common multiple $m\in R$ of $a$ and $b$ is necessarily a multiple of $l$, i.e. $l\mid m$.
\end{enumerate}
\end{definition}

Gcds, when they exist, are only defined \textbf{up to associates}, i.e. only up to multiplication by a unit $u\in
R^{\times}$. Explicitly, for $a,b\in R$, I will write $a\sim_R b$ if $\exists u\in R^{\times}: ua=b$. I
will simply write $a\sim b$ when the ring is clear from the context. It is a standard result that this is an
equivalence relation, i.e. it is reflexive, symmetric, and transitive. I assume that $\gcd(a,b)$ is some fixed choice
among all valid associates.
Consequently, for $u,v\in R^{\times}$, we have $\gcd(ua,vb)\sim\gcd(a,b)$, but equality need not hold. 
Lcms behave similarly.

Motivated by a question on Mathematics Stack Exchange by Ash Seifert~\cite{Sei22}, I define the following notions,
which I could not find in the literature despite their seemingly obvious nature.

\begin{definition}\label{def:lcmSplit}
Let $R$ be a GCD ring and $a,b\in R$ two elements. 
\begin{enumerate}
\item[(a)] An \textbf{lcm splitting} of $a,b$, if it exists, consists of a pair $c,d\in R$ such that $cd\sim\lcm(a,b),
  c\mid a, d\mid b$ and $\gcd(c,d)\sim1$. The fourth and last condition is called the \textbf{splitting condition}.
\item[(b)] $R$ is \textbf{lcm splitting} if every pair of elements admits an lcm splitting.
\end{enumerate}
\end{definition}

The main contributions of this paper are
\begin{itemize}
\item the definitions of lcm splitting just given,
\item a theorem that adequate (defined below) implies lcm splitting for domains, and
\item two explicit examples of B\'{e}zout domains that are not lcm splitting.
\end{itemize}
Some of the presented ideas were developed in dialog with chatbots, cf. Sec.~\ref{sec:ackChatbots}.

\bigskip
\section{Preliminaries, related notions and results}
\label{sec:Preliminaries}

Most rings considered will be \textbf{integral domains}, i.e., they have no nontrivial zero divisors. A \textbf{GCD domain} is
a GCD ring that is an integral domain. Domains are canonically embedded in a \textbf{fraction field
  $R\subseteq\Frac(R)$}. This allows one to write certain elements of $R$ as fractions with nonzero denominator, simplifying
notation. E.g., if $\gcd(a,b)\neq 0$ and $a'\in R$ satisfies $a=a'\gcd(a,b)$, then the \textbf{fraction
  $a'=\frac{a}{\gcd(a,b)}\in\Frac(R)$} is well-defined and, moreover, lies in $R$.

In Prop.~\ref{prop:gcdIdentities} in the appendix, I collect some standard identities for gcd and lcm for $a, b, c$
elements of a GCD domain $R$, assumed to be nonzero when in doubt. Most references deal with the special case
$R=\mathbb{Z}$, so I will supply (references to) proofs in the general context to be self-contained.

\begin{remark}\label{rmk:multipleGCD}
I will apply the identities and properties of Prop.~\ref{prop:gcdIdentities} by labeling steps with the corresponding numbers
throughout the paper. Numbers in parentheses in equation-like statements without further explanation will always refer to
Prop.~\ref{prop:gcdIdentities}.

Associativity (7) lets us define \textbf{multiple gcds and lcms} iteratively. E.g. $\gcd(a,b,c)=\gcd(\gcd(a,b),c)$ or
$\lcm(a,b,c,d)=\lcm(\lcm(\lcm(a,b),c),d)$. 
\end{remark}

Next, I recall some classical notions used in the subsequent sections, followed by references for details. Then I recall
several related facts, introducing further classical terminology.

\begin{definition} \label{def:GCDproperties}
\begin{enumerate}
\item A GCD ring $R$ is a \textbf{B\'{e}zout ring} if all its finitely generated ideals are principal. In
  particular (and ultimately equivalently), for any $a,b\in R$ the ideal identity $(a,b)=(\gcd(a,b))$ holds. In general,
  only ``$\subseteq$'' is guaranteed. In the language of the generalized Euclidean algorithm, the ideal identity (and
  thus the B\'{e}zout property) can equivalently be stated as $\exists c,d\in R: ad-bc=\gcd(a,b)$.\\
A \textbf{B\'{e}zout domain} is a B\'{e}zout ring that is also an integral domain.
\item \label{def:gf} A GCD ring $R$ has \textbf{good factorization (gf)} if for all $a, b \in R\setminus\{0\}$ there exist
  $c,d\in R$ with $a = cd, \gcd(c,b)\sim1$ and $b\in\sqrt{(d)}$.
\item A ring $R$ is \textbf{adequate} if for all $a,b\in R\setminus\{0\}$ there exist $c,d\in R$ with $a=cd,
  bR+cR=R$ and, for all $d'\in R$, the inclusions $dR\subseteq d'R\subsetneq R$ imply $bR+d'R\subsetneq R$.\\
If $R$ is a B\'{e}zout ring, this is equivalent to: For all $a,b\in R\setminus\{0\}$ there exist
$c,d\in R$ with $a=cd, \gcd(c,b)\sim1$ and for all non-units $d'\in R\setminus R^{\times}$ the divisibility $d'\mid d$
implies $\gcd(d',b)\not\sim 1$. Note the similarity to the definition of gf.
\item For an $m\times n$-matrix $A\in\Mat_{m,n}(R)$ with coefficients in a ring $R$ ($m,n\in\mathbb{N}_{>0}$),
  we say that $A$ admits a \textbf{diagonal reduction} if there are invertible (over $R$) matrices $P\in\Mat_{m,m}(R),
  Q\in\Mat_{n,n}(R)$ such that $PAQ=\Diag((d_i)_{i=1}^{\min(m,n)})$ is diagonal with $d_i\mid d_{i+1}$ for all 
  $1\leq i\leq\min(m,n)-1$. $R$ has the \textbf{elementary divisor} property if all (finite-dimensional) matrices admit diagonal reductions.
\item For $A\in\Mat_{m,n}(R)$ over a ring $R$ ($m,n\in\mathbb{N}_{>0}$), we say that $A$ admits a
  \textbf{triangular reduction} if there is an invertible (over $R$) matrix $Q\in\Mat_{n,n}(R)$ such that
  $AQ$ is lower triangular (arbitrary entries on the diagonal are allowed). $R$ is \textbf{K-Hermite} if all
  (finite-dimensional) matrices admit triangular reductions. 
\end{enumerate}
\end{definition}

I will say that a pair $(a,b)$ satisfies resp. violates gf or adequate if the respective property holds resp. fails for $(a,b)$.
Anderson~\cite{And00} provides a survey of \textbf{divisibility theory for GCD domains}. Cohn~\cite{Coh68} gives results
about \textbf{B\'{e}zout domains} (Cohn calls GCD domains HCF domains and B\'{e}zout
domains B\'{e}zout rings). Gillman and Henriksen~\cite{GH56A} (they call B\'{e}zout rings F-rings and
K-Hermite Hermite), Lam~\cite[Ch. 1, Def. 4.23]{Lam06} and L'Innocente and Point~\cite[Sec. 5.2]{LP20} elaborate on the
properties \textbf{K-Hermite, elementary divisor, adequate and good factorization}. A textbook by
Cohn~\cite[Ch. 6.5]{Coh74} discusses some elementary properties of divisibility, gcds and lcms.

In the literature, gf and adequate are often used only in the narrower context of B\'{e}zout domains. Quoting van den
Dries and Macintyre~\cite[Sec. 2.10]{vDM90} and adjusting notation: ``In intuitive terms, writing \textit{prime factor of
  $x$} for \textit{maximal ideal containing $x$}, this implies that the prime factors of $c$ are those of $a$ that do
not occur in $b$, while the prime factors of $d$ are exactly the common prime factors of $a$ and $b$.''

Good factorization resp. adequate allow one to split a ring element into factors multiplicatively well aligned with
another ring element. This is close in spirit to the idea of an lcm splitting. Indeed, in Sec.~\ref{sec:adequLcmSplitting} 
I prove adequate (and gf) implies lcm splitting. \cite[Sec. 2.10]{vDM90} also sketches an
example, due to R. Wiegand, in which gf fails. In Sec.~\ref{sec:counterexampleLCMSplitting}, this is my starting point
for the first counterexample: a B\'{e}zout domain that is not lcm splitting. Note
that gf and adequate are not symmetric: $(a,b)$ can be gf resp. adequate while $(b,a)$ is not, cf. 
Prop.~\ref{prop:RIsNotAdequate}. However, the definition of lcm splitting is clearly symmetric.

\newpage
\begin{remark}\label{rmk:propertiesOfGCDDomains}
To be self-contained, I give an overview of standard results on GCD domains.
\begin{itemize}
\item In GCD domains, lcms always exist, cf. \cite[Thm 2.1]{Coh68} or \cite[Thm 2.1 (3)]{And00} for general
  statements about integral domains.
\item GCD domains that satisfy the \textbf{ascending chain condition for principal ideals (ACCP)} are \textbf{UFDs
  (unique factorization domains)}. In particular, all \textbf{Noetherian} GCD domains are UFDs. However, there are
  non-Noetherian UFDs. In general, we have the following sequence of proper implications (meaning that there are
  counterexamples to their converses) between ring properties:
\[
\text{ UFD }\Rightarrow
\text{ GCD domain }\Rightarrow
\textbf{ Schreier domain }\Rightarrow
\textbf{ integrally closed domain }\Rightarrow
\text{ domain.}
\]
\item 
B\'{e}zout domains are exactly those GCD domains that are \textbf{Pr\"{u}fer} (i.e. all finitely generated ideals are
  invertible). Any integral domain is the intersection of its \textbf{localizations} at maximal ideals resp. prime ideals. Moreover,
\begin{gather*}
R \text{ is a Prüfer ring }\Leftrightarrow\\
\forall \mathfrak{p} \in\spec(R): R_{\mathfrak{p}} \text{ is a valuation ring }\Leftrightarrow\\
\forall \mathfrak{m} \in\mspec(R): R_{\mathfrak{m}} \text{ is a valuation ring.}
\end{gather*}
For a B\'{e}zout domain $R$ and a maximal ideal $\mathfrak{m}$, the \textbf{valuation} of $R_{\mathfrak{m}}$ restricts to
a valuation $v_{\mathfrak{m}}$ on $R$. Valuations of this type characterize $\gcd$ and $\lcm$ up to units:
\[
v_{\mathfrak{m}}(\gcd(a,b))=
\min\left(v_{\mathfrak{m}}(a),v_{\mathfrak{m}}(b)\right) \text{ and }
v_{\mathfrak{m}}(\lcm(a,b))=
\max\left(v_{\mathfrak{m}}(a),v_{\mathfrak{m}}(b)\right).
\]
\item For a B\'{e}zout domain $R$, the references above give the implications
\[
\text{UFD }\Rightarrow
\text{ gf }\stackrel{\text{\tiny \cite[p. 460]{LP20}}}{\Rightarrow}
\text{ adequate }\stackrel{\text{\tiny \cite[p. 86]{RRW24}}}{\Rightarrow}
\text{ elementary divisor }\stackrel{\text{\tiny \cite[p. 363]{GH56A}}}{\Rightarrow}
\text{ K-Hermite }\stackrel{\text{\tiny \cite[ch. 1, Cor. 4.28]{Lam06}}}{\Leftrightarrow}
\text{ B\'{e}zout domain.}
\]
The implication ``UFD $\Rightarrow$ gf'' is an almost immediate consequence of comparing the prime factorizations of
$a$ and $b$, constructing $d$ from the prime factors of $a$ that appear in $b$, and $c$ from those that do not. For ``gf
$\Rightarrow$ adequate'', see also Lem.~\ref{lem:gfImpliesAdequateForPairs}. Note that some of the
implications hold under weaker assumptions than $R$ being a B\'{e}zout domain, in particular the last equivalence. See
the references for the proper level of generality. Some works study these properties for general rings,
which need not be commutative or domains.
Note that diagonal reduction needs to be checked only for matrix dimensions $(m,n)\in\{(1,2), (2,1), (2,2)\}$ for the
elementary divisor property to hold, cf. \cite[Thm. 2]{GH56A}. Similarly, K-Hermite needs to be checked only for
$(m,n)\in\{(1,2)\}$, cf. \cite[ch. 1, Prop. 4.24]{Lam06}.

A classic example of a non-UFD, gf domain is the ring of algebraic integers $R_3=\overline{\mathbb{Z}}$,
cf. App.~\ref{sec:furtherExamples}. I was not able to find an example of a non-gf, adequate domain in the
literature. 
I also do not know whether there are non-adequate, elementary divisor domains that are lcm splitting.
In Sec.~\ref{sec:counterexampleLCMSplitting}, I will give
two examples of non-adequate, elementary divisor domains that are not lcm splitting. Recently, H\"{a}gg and
M\"{o}rtberg~\cite{HM26} announced an example of a B\'{e}zout domain that is not an elementary divisor domain -- I do
not know whether it is lcm splitting.

Let $R$ be a B\'{e}zout domain and $a,b\in R\setminus\{0\}$. If a pair $(c,d)\in R^2$ satisfies the gf
resp. adequate conditions for $(a,b)$, then it is unique up to associates, cf. Cor.~\ref{cor:uniquenessInGFSplitting}
and Lem.~\ref{lem:uniquenessInAdequateSplitting}.

\item General GCD domains may be \textbf{inertly embedded} into B\'{e}zout domains, i.e. without changing the
  divisibility and thus $\gcd$-properties too much, see \cite{And00} and \cite{Coh68}. This suggests, on an informal
  level, that most questions about divisibility for GCD domains may already be answered through B\'{e}zout domains.
\end{itemize}
\end{remark}

\begin{remark}\label{rem:firstObsLCMSplitting}
I give some first observations on the lcm splitting property. I remind the reader of the convention in
Rem.~\ref{rmk:multipleGCD} about numbers in parentheses.
\begin{itemize}
\item In Def.~\ref{def:lcmSplit}, replacing $cd\sim\lcm(a,b)$ by $cd=\lcm(a,b)$ does not change the concept.
\item For $a=0$ or $b=0$ (or both), lcm splitting is always satisfied. Assume $a=0$ and set $c=0,d=1$. Then $cd=0\cdot
  1=0\stackrel{\text{\tiny (4), (6)}}{=}\lcm(0,b)=\lcm(a,b)$, $c=0\mid 0=a$, $d=1\mid b$ since $1$ divides every ring element,
  and $\gcd(c,d)=\gcd(0,1)\stackrel{\text{\tiny (4)}}{\sim}1$. Similarly for $a=0=b$. We can wlog assume that
  $a\neq0\neq b$. E.g., below $\gcd(a,b)\neq 0$ is assumed in terms like $\frac{b}{\gcd(a,b)}$.
\item LCM splittings need not be unique, not even up to associates. E.g., in $R=\mathbb{Z}$, for $a=2\cdot 3=6, b=2\cdot 5=10$,
  any of $(c,d)\in\{(\pm6,\pm5), (\pm3,\pm10)\}$ is an lcm splitting.
\item Without the splitting condition, the other three properties can always easily be satisfied, e.g.,
  by $(c,d)=\left(a,\frac{b}{\gcd(a,b)}\right)$. Use the product formula (5) for the first condition. The two
  divisibility conditions are trivially satisfied. The following example in $R=\mathbb{Z}$ (with gcds chosen nonnegative)
  demonstrates that this approach does not guarantee the splitting condition: for $a=2^2\cdot 3=12, b=2\cdot 3^2=18$, we
  have $\frac{b}{\gcd(a,b)}=\frac{18}{6}=3$ and thus $\gcd\left(a,\frac{b}{\gcd(a,b)}\right)=\gcd(12,3)=3\not\sim1$.
\item One could try to iterate this division by common factors until the evolving pair $(c,d)$
  becomes coprime. For this, start with $(c_0,d_0)=\left(a,\frac{b}{\gcd(a,b)}\right)$. If $\gcd(c_0,d_0)\not\sim1$,
  intuitively speaking, $a$ and $b$ have a ``common \textbf{irreducible (atomic)} factor'' that has higher multiplicity
  in $b$ than in $a$, so it is present in $a$, but also ``survives'' the division by $\gcd(a,b)$ in $d_0$. The idea is to shift
  enough factors back from $c_i$ to $d_i$ until $(c_i,d_i)$ become coprime. For this, define
  $(c_{i+1},d_{i+1})=\left(\frac{c_i}{\gcd(c_i,d_i)},d_i\cdot\gcd(c_i,d_i)\right)$ -- note that this also makes sense
  if $(c_i,d_i)$ are already coprime.\label{page:lcmSplittingAlgo}\\
This algorithm\footnote{Technically, whether this effective procedure is an algorithm depends on the domain $R$. Recall
that most definitions of an algorithm require guaranteed termination, which is not available here in general.} gives us an lcm splitting if the sequence of $(c_i,d_i)$ stabilizes after finitely many steps (a rigorous proof of
this claim is left to the reader). This is guaranteed, for instance, if the ascending chain of principal ideals
$(c_0)\subseteq(c_1)\subseteq\ldots\subseteq(c_i)\subseteq\ldots$ always stabilizes, i.e. if the ACCP holds for
enough sequences. As discussed above, GCD domains satisfying the ACCP are
exactly UFDs. Hence \textbf{UFDs are lcm splitting}. The ring of entire functions $R_2=\mathcal{O}(\mathbb{C})$ is an lcm
splitting domain, for which this procedure does not terminate after finitely many steps, cf. App.~\ref{sec:furtherExamples}.
\item We can see that UFDs are lcm splitting in a different way, using the fact that UFDs admit a unique factorization into
  irreducible (atomic) elements as was intuitively alluded to before. For $a,b$, such unique factorizations can be written
  in the following form:
\[
a=\prod_{i=1}^lp_i^{e_i}\cdot \prod_{i=1}^mq_i^{f_i}\cdot \prod_{i=1}^nr_i^{g_i},\quad
b=\prod_{i=1}^lp_i^{e_i'}\cdot \prod_{i=1}^mq_i^{f_i'}\cdot \prod_{i=1}^nr_i^{g_i'},
\]
where all the $p_i,q_i,r_i$ are pairwise coprime, $e_i>e_i'\geq0$, $f_i=f_i'>0$, and $0\leq g_i<g_i'$. The atoms $p_i$ need to go to
$c$, the $r_i$ to $d$, and the (full power of each) $q_i$ can be distributed arbitrarily. Two
possibilities for an lcm splitting (that coincide if $m=0$) are given by
\[
(c,d)=\left(\prod_{i=1}^lp_i^{e_i}\cdot \prod_{i=1}^mq_i^{f_i},\quad\prod_{i=1}^nr_i^{g_i'}\right),\quad
(c',d')=\left(\prod_{i=1}^lp_i^{e_i},\quad\prod_{i=1}^mq_i^{f_i'}\cdot \prod_{i=1}^nr_i^{g_i'}\right).
\]
It is easily verified that these pairs constitute lcm splittings. This is a direct way to see that UFDs are lcm
splitting. However, the other approach also has its merits. It gives an algorithm for UFDs whenever gcds are computable,
which is often the case, while a direct algorithm for unique factorization is often unavailable or more expensive.
\end{itemize}
\end{remark}

\bigskip
\section{Adequate implies lcm splitting}
\label{sec:adequLcmSplitting}

I remind the reader of the convention in Rem.~\ref{rmk:multipleGCD} about numbers in parentheses.

\begin{lemma}\label{lem:adequateFactorsAreCoprime}
Let $R$ be a GCD domain. Given $a, b\in R\setminus\{0\}$ and some $c, d \in R$ satisfying the adequate conditions,
i.e. $a = cd, \gcd(c,b)\sim1$ and $\forall d'\in R\setminus R^{\times}:d'\mid d\Rightarrow \gcd\left(d',b\right)\not\sim
1$, we have $\gcd(c,d)\sim1$.
\end{lemma}

\begin{proof}
Assume for contradiction that $g=\gcd(c,d)\not\sim1$. By the definition of the gcd, $g=\gcd(c,d)\mid d$. Hence the
adequate property for $d'=g\not\sim1$ implies $\gcd(g,b)\not\sim1$. But also $g=\gcd(c,d)\mid c$, so
\[
\gcd(g,b)=
\gcd(\gcd(c,d),b)\stackrel{\text{\tiny (6), (11)}}{\mid}
\gcd(c,b)\sim
1 \stackrel{\text{\tiny (9)}}{\Rightarrow}
\gcd(g,b)\sim
1.
\] 
This contradiction shows that the assumption is false. Thus $g=\gcd(c,d)\sim1$.
\end{proof}

\begin{theorem} \label{thm:adequateImpliesLCMSplitting}
An adequate B\'{e}zout domain $R$ is lcm splitting.
\end{theorem}

\begin{proof}
Let $a, b \in R\setminus\{0\}$ be arbitrary (cf. Rem.~\ref{rem:firstObsLCMSplitting}). Set
$g=\gcd(a,b), l=\lcm(a,b)$. From (5), $ab\sim gl$, we deduce $l\sim\frac{ab}{g}=\frac{a}{g}g\frac{b}{g}$.

Now apply the adequate property to the pair $\left(g, \frac{b}{g}\right)$ to get $g=g_1g_2$ with
$\gcd\left(g_1,\frac{b}{g}\right)\sim1$ and $\forall \overline{g}_2\in R\setminus R^{\times}:\overline{g}_2\mid
g_2\Rightarrow \gcd\left(\overline{g}_2,\frac{b}{g}\right)\not\sim 1$. 
Apply adequate a second time to $\left(g_2, \frac{a}{g}\right)$ to get $g_2=g_1'g_2'$ with
$\gcd\left(g_2',\frac{a}{g}\right)\sim1$ and $\forall \overline{g}_1\in R\setminus R^{\times}:\overline{g}_1\mid
g_1'\Rightarrow \gcd\left(\overline{g}_1,\frac{a}{g}\right)\not\sim 1$. 
Overall, we have $l\sim\frac{a}{g}g\frac{b}{g}=\frac{a}{g}g_1g_1'g_2'\frac{b}{g}$, and we will show some coprimality
properties among these factors:
\begin{enumerate}
\item[(A)] $\gcd\left(\frac{a}{g},g_2'\right)\sim1$ follows directly from the 2nd application of adequate and (6). 
\item[(B)] Setting $c=g$ in (13), we get $\gcd\left(\frac{a}{g},\frac{b}{g}\right)\sim\frac{\gcd(a,b)}{g}=\frac{g}{g}=1$. 
\item[(C)] 
$\gcd(g_1,g_2')\stackrel{\text{\tiny (11)}}{\mid}
\gcd(g_1,g_1'g_2')=
\gcd(g_1,g_2)\sim 1$, 
where the last step follows from Lem.~\ref{lem:adequateFactorsAreCoprime}. By (9) we get $\gcd(g_1,g_2')\sim1$.
\item[(D)] $\gcd\left(g_1,\frac{b}{g}\right)\sim1$ follows directly from the 1st application of adequate.
\item[(E)] Assume $1\not\sim g_1'$. From $g_1'\mid g_1'$ and the 2nd application of adequate (setting
  $\overline{g}_1=g_1'$), we have $f=\gcd\left(g_1',\frac{a}{g}\right)\not\sim 1$.\\
Since $1\not\sim f=\gcd\left(g_1',\frac{a}{g}\right)\mid g_1'\mid g_1'g_2'=g_2$, the 1st application of adequate
(setting $\overline{g}_2=f$) yields $f'=\gcd\left(f,\frac{b}{g}\right)\not\sim 1$.\\
But $f'=\gcd\left(f,\frac{b}{g}\right)\mid\frac{b}{g}$ and $f'=\gcd\left(f,\frac{b}{g}\right)\mid
f=\gcd\left(g_1',\frac{a}{g}\right)\mid\frac{a}{g}$.\\
By (1), we get $1\not\sim f'\mid\gcd\left(\frac{a}{g},\frac{b}{g}\right)\stackrel{\text{\tiny (B)}}{\sim}1$,
which by (9) is a contradiction. Therefore $g_1'\sim1$.
\end{enumerate}

By (E), only four nontrivial factors remain: $l\sim\frac{a}{g}g_1g_2'\frac{b}{g}$. We set $c=\frac{a}{g}g_1$ and 
$d=g_2'\frac{b}{g}$. Now we apply (12) several times based on the coprimality results (A)--(D):
\[
(\alpha) \gcd\left(\frac{a}{g},g_2'\right)\stackrel{\text{\tiny (A)}}{\sim}1 \stackrel{\text{\tiny (12)}}{\Rightarrow}
\gcd\left(\frac{a}{g},g_2'\frac{b}{g}\right)\sim
\gcd\left(\frac{a}{g},\frac{b}{g}\right)\stackrel{\text{\tiny (B)}}{\sim}
1,
\]
\[
(\beta) \gcd\left(g_1,g_2'\right)\stackrel{\text{\tiny (C)}}{\sim}1 \stackrel{\text{\tiny (12)}}{\Rightarrow}
\gcd\left(g_1,g_2'\frac{b}{g}\right)\sim
\gcd\left(g_1,\frac{b}{g}\right)\stackrel{\text{\tiny (D)}}{\sim}
1,
\]
\[
(\gamma) \gcd\left(\frac{a}{g},g_2'\frac{b}{g}\right)\stackrel{\text{\tiny $(\alpha)$}}{\sim}1 \stackrel{\text{\tiny (6), (12)}}{\Rightarrow}
\gcd\left(c,d\right)=
\gcd\left(\frac{a}{g}g_1,g_2'\frac{b}{g}\right)\sim
\gcd\left(g_1,g_2'\frac{b}{g}\right)\stackrel{\text{\tiny $(\beta)$}}{\sim}
1.
\]
Having shown $\gcd\left(c,d\right)\sim1$, we verify three further properties.
$cd=\frac{a}{g}g_1g_2'\frac{b}{g}\sim l=\lcm(a, b)$ by construction. From $a=\frac{a}{g}g=\frac{a}{g}g_1g_2$, we get
$c=\frac{a}{g}g_1\mid a$, and finally, from $b=g\frac{b}{g}=g_1g_1'g_2'\frac{b}{g}$, we get $d=g_2'\frac{b}{g}\mid b$.
Since $a,b$ were arbitrary, $R$ is an lcm splitting B\'{e}zout domain.
\end{proof}

\begin{remark}\label{rem:effectiveLCMSplitting}
An effective method for computing the pair $(c,d)$ in the adequate property would yield an effective method for computing lcm
splittings. By Thm.~\ref{thm:gfImpliesLCMSplitting}, the same holds for gf.
\end{remark}

\bigskip
\section{B\'{e}zout domains that are not lcm splitting}
\label{sec:counterexampleLCMSplitting}

\subsection{First example}
\label{sec:counterexampleLCMSplitting1}

I will give an explicit B\'{e}zout domain $R$ and elements $a,b\in R\setminus\{0\}$ that fail lcm splitting. This first 
example is based on a B\'{e}zout domain sketched by van den Dries and Macintyre~\cite[Sec. 2.10]{vDM90}, which they
attribute to Wiegand. The example exhibits a B\'{e}zout domain $R$ that is not gf, which by
Thm.~\ref{thm:adequateImpliesLCMSplitting} and ``gf $\Rightarrow$ adequate'' (cf. Rem.~\ref{rmk:propertiesOfGCDDomains}) is a
requirement for a counterexample to lcm splitting. The example relies on several constructions of medium
complexity from the literature, which are worth making explicit. The failure of gf can be established essentially by
topological properties, so \cite[Sec. 2.10]{vDM90} can avoid going into details about the constructions, essentially
giving the gist of what is needed in a short paragraph. For lcm splitting, one has to dig deeper. As it is quite
illustrative, I will present a failed attempt at a counterexample to lcm splitting. I hope this sheds some light on the
property and serves as a partial guide to constructing further examples of failure of lcm splitting.

As a starting point, I use a construction of B\'{e}zout domains with prescribed topological data; it proceeds in
several steps scattered across the literature. Outline:
\begin{enumerate}
\item[(A)] Fischer~\cite{Fis79} gives criteria for when a topological space can be the \textbf{maximal spectrum
  $\mspec(R)$} of a B\'{e}zout domain $R$, and presents constructions of \textbf{spectral trees} that encode the
  \textbf{(prime) spectrum $\spec(R)$} of a B\'{e}zout domain $R$ together with topological information on $\spec(R)$.
\item[(B)] Fischer~\cite{Fis78} earlier gave criteria for when a topological space can be the (prime) spectrum $\spec(R)$
  of a B\'{e}zout domain $R$, and gave a construction of a \textbf{lattice ordered (abelian\footnote{The groups in this
    construction will always be abelian, so I will not mention this again.}) group $G$}, whose associated
  \textbf{group spectrum} $\spec(G)\cong(\spec(R)\setminus\{(0)\})$ is homeomorphic to the ring spectrum except for the
  generic point\footnote{Fischer~\cite{Fis78} gives constructions for more general rings that need not be
  domains. I focus on domains.}.
\item[(C)] These lattice ordered groups $G$ serve as \textbf{demivalue groups} and \textbf{groups of divisibility} for a
  B\'{e}zout domain $R$. Jaffard~\cite[Ch. III.2.2, pp. 78 ff.]{Jaf60} gave a construction of a \textbf{$k$-algebra $R$},
  for any field $k$, that is a B\'{e}zout domain\footnote{Jaffard originally only asserted it to be a GCD
  domain. Ohm~\cite{Ohm69} later showed that the construction actually gives a B\'{e}zout domain.} with the desired
  $\spec(R)\setminus\{(0)\}\cong\spec(G)$ as topological spaces. An English textbook version of this part of the
  construction was given by Gilmer~\cite[Thm. 18.6, pp. 215 ff.]{Gil92}.
\end{enumerate}

In the construction of a non-gf B\'{e}zout domain in \cite[Sec. 2.10]{vDM90}, the goal is a topological space in which
strong separation properties fail. The idea is to have a countable set of isolated points with two distinct but
inseparable ``\textbf{limit points}'' attached to it. 
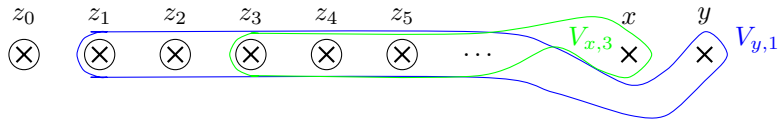
\begin{figure}[!h]
\centering
\begin{tikzpicture}[croix/.pic={\draw[thick,black] (-0.1,-0.1) -- (0.1,0.1) (-0.1,0.1)--(0.1,-0.1);}]
\foreach \z in {0,1,...,5}{
\draw (\z,0) pic {croix} node[above=8]{$z_{\z}$};
\draw (\z,0) circle [radius=0.2];}
\draw (6,0) node{$\ldots$};
\draw (8,0) pic {croix} node[above=8]{$x$};
\draw (9,0) pic {croix} node[above=8]{$y$};
\draw[color=blue] (9.7,0.2) node{$V_{y,1}$};
\draw[color=green] (7.5,0.2) node{$V_{x,3}$};
\draw[color=blue] plot [smooth cycle] coordinates {(9.3,0) (9,0.3) (8.5,-0.2) (8,-0.4) (7,0.1) (6,0.3) (1.2,0.3) (1.1,0.3) (1,0.3) (0.79,0.21) (0.7,0) (0.79,-0.21) (1,-0.3) (1.1,-0.3) (1.2,-0.3) (6,-0.3) (7,-0.5) (8,-0.8) (8.5,-0.8) (9,-0.35)};
\draw[color=green] plot [smooth cycle] coordinates {(8.3,0) (8,0.3) (7.5,0.5) (6,0.28) (3.2,0.28) (3.1,0.28) (3,0.28) (2.81,0.19) (2.72,0) (2.81,-0.19) (3,-0.28) (3.1,-0.28) (3.2,-0.28) (6,-0.28) (7,0.1) (7.6,-0.2) (8,-0.3)};
\end{tikzpicture}
\caption{\footnotesize The space $V$ and several open sets (circles and ``tubes'') in the spectral topology.}
\label{fig:VSpecTop}
\end{figure}
\begin{definition} \label{def:spaceV}
Explicitly, $V=\{z_i:i\in\mathbb{N}\}\dot{\cup}\{x,y\}$ carries the spectral topology in which each $z_i$ is an isolated point
and a basis of open neighborhoods of $x$ resp. $y$ is given by $V_{m,i}=\{z_n:n\geq i\}\cup\{m\}$ for each
$i\in\mathbb{N}, m\in\{x,y\}$ separately.
\end{definition}
I illustrate the topology in Fig.~\ref{fig:VSpecTop} and describe it explicitly. A subset $\{z_i: i\in Z\}\cup
L\subseteq V$ for $Z\subseteq\mathbb{N}, L\subseteq\{x,y\}$ in the \textbf{spectral topology on $V$} is
\begin{itemize}
\item \textbf{closed} iff ($L=\{x,y\}$ and $Z$ arbitrary) or ($L\neq\{x,y\}$ and $Z$ \textbf{finite}),
\item \textbf{open} iff ($L=\emptyset$ and $Z$ arbitrary) or ($L\neq\emptyset$ and $Z$ \textbf{cofinite}),
\item \textbf{quasi-compact open} iff ($L=\emptyset$ and $Z$ finite) or ($L\neq\emptyset$ and $Z$ cofinite).
\end{itemize}

Fischer~\cite[Prop. 2.1]{Fis79} shows that the maximal spectrum of any B\'{e}zout domain is a \textbf{B-space}, so I
repeat the definition here. In fact, the whole construction will make use of several topological notions; for easier
reference, I define the less familiar ones here.
\begin{definition} \label{def:topPropertiesBSpace}
\begin{itemize}
\item A \textbf{spectral space} is a topological space homeomorphic to a ring spectrum $\spec(R)$ for some ring $R$. All
  spectral spaces are sober, in particular T0.
\item Given any topology (e.g. a spectral topology), the associated \textbf{patch topology} has as a basis of closed sets
  all sets that are either (a) closed in the original topology or (b) quasi-compact open in the original topology. A
  closed set in the patch topology is called a \textbf{patch}. Often patches are not viewed with the induced subspace patch
  topology, but with the subspace original topology.
\item We define the \textbf{specialization order} for points $x,y\in X$ of a topological space by 
\[x\geq_s y \Leftrightarrow x\in\overline{\{y\}}^{\text{cl}}.\]
$\geq_s$ is a partial order and for $x\geq_sy$ we say $x$ is a \textbf{specialization} of $y$ resp. $y$ is a
\textbf{generization} of $x$. The specialization order adds no new information to a topology, but sometimes allows for
more concise descriptions. We set the \textbf{generization set for $x$} to be $G(x)=\{y\in X:x\geq_s y\}$. If, for all
$x\in X$, the generization set $G(x)$ is totally ordered by the specialization order, we say that $X$ is a
\textbf{tree}\footnote{Note that a tree in this context does not need to be connected in 
the sense that each pair of points would have a common generization as ``ancestor'' or 
have a unique root. It could have several 
``roots'' (elsewhere called a forest), or none at all (generization sets could have infinitely long descending sequences
without an infimum being realized by a smallest element). The usual combinatorial trees
are a different notion as there are no ``edges''.}. A spectral space that is a tree is called a \textbf{spectral tree}.
\item A topological basis of quasi-compact open sets is called a \textbf{QCI basis (quasi-compact with intersection
  property)} if the intersection of any finitely many members of the basis again lies in the basis.
\item The elements of the power set $\mathcal{P}(X)$ form the Boolean algebra of subsets of $X$ (with operations union,
  intersection and complement). Consider the Boolean subalgebra $A(X)\subseteq\mathcal{P}(X)$ generated by quasi-compact
  open sets for the (original) topology on $X$. We call elements of $A(X)$ \textbf{constructible}. If $X$ is a spectral
  tree, the constructible sets (for the spectral topology) are exactly the \textbf{clopen (closed and open)} subsets of
  $X$ with respect to the patch topology.
\item Let $X$ be a topological space. For every integer $n\geq 1$, a sequence $(x_1, \ldots, x_n) \in X^n$ is said to be 
\textbf{sticky} provided $\bigcap_{i=1}^nN_i\neq\emptyset$ whenever $N_i$ is a neighborhood of $x_i$. A quasi-compact 
T1 space is called a \textbf{B-space} provided every closed subspace $Y$ has the following property: If 
$(y_{i-1},y_i)\in Y^2$ is sticky for $2\leq i\leq n$, then $(y_1, \ldots, y_n)$ is sticky.
\end{itemize}
\end{definition}
I restate Fischer~\cite[Prop. 3.3]{Fis79} for a topological space $X$:
\begin{prop}\label{prop:Fischer79_3.3}
Suppose $X$ has only finitely many non-isolated points. Then $X$ is homeomorphic to the maximal ideal space of a 
finite-dimensional B\'{e}zout domain if and only if 
\begin{enumerate}
\item $X$ is a B-space, and 
\item $X$ has an open basis consisting of quasi-compact sets.
\end{enumerate}
\end{prop}

\newpage
My intermediate aim is to follow \cite[Sec. 2.10]{vDM90} and make the failure of gf explicit. We need
\begin{lemma}\label{lem:basicClosedSetsInBezoutSpectra}
Let $R$ be a B\'{e}zout ring and $X=\mspec(R)$ its maximal spectrum as a topological space. If $C\subset X$ is closed
with quasi-compact complement $X\setminus C$, then $C=Z(r)$ is the zero set of a single ring element $r\in R$ (resp. of
its principal ideal $(r)$).
\end{lemma}
\begin{proof}
In the Zariski topology (for the maximal spectrum as well as the prime spectrum), the complements of zero sets of single ring elements 
$D(s)=X\setminus Z(s)$ form an open basis. In a B\'{e}zout ring where any sum of two principal ideals is again principal, 
we have $(a,b)=(a)+(b)=(g)$ for $g=\gcd(a,b)$; note that this requires exactly B\'{e}zout and will fail for general GCD rings. 
In particular, the union of two basic open sets is again a basic open set:
\[
D(a) \cup D(b)=
(X\setminus Z(a)) \cup (X\setminus Z(b))=
X\setminus (Z(a) \cap Z(b))=
X\setminus Z(a,b)=
X\setminus Z(g)=
D(g).
\]
By induction, for any finite collection $(a_i)_{i=1}^n\in R^n$, we can find an $r\in R$ such that $\bigcup_{i=1}^nD(a_i)=D(r)$.

Now we take the closed $C$ from the assumptions. Its complement $X\setminus C$ is quasi-compact open. Since the sets of the form 
$D(a)$ are an open basis in the spectral topology, there is some index set $\mathcal{I}$ and a family of open $D(a_i)\subseteq X\setminus C$, one for each 
$i\in\mathcal{I}$, such that their union $\bigcup_{i\in\mathcal{I}}D(a_{i}) = X\setminus C$. Since $X\setminus C$ is 
quasi-compact, we can choose a finite subset of $\mathcal{I}$ to do the job, i.e., after relabeling we can assume 
$\mathcal{I}=\{1,\ldots,n\}$ for some $n\in\mathbb{N}$ and $X\setminus C = \bigcup_{i=1}^nD(a_i) = D(r)$, where the last equality 
uses the B\'{e}zout ring property derived in the previous paragraph. But $X\setminus C = D(r)$ is by definition equivalent to 
$C=Z(r)$. This finishes the proof.
\end{proof}

\begin{prop}\label{prop:spaceVFailureOfGF}
For the space $V$ from Def.~\ref{def:spaceV} there is a B\'{e}zout domain $R_V$ with $\mspec(R_V)\cong V$ such that gf fails
for some $\tilde{a},\tilde{b}\in R_V\setminus\{0\}$. In particular, gf fails in some B\'{e}zout domain.
\end{prop}
\begin{proof}
We check that $V$ satisfies all requirements of Prop.~\ref{prop:Fischer79_3.3}.
\begin{enumerate}
\item $V$ is a T1 space. Any $z_i$ is isolated, which proves separation for any pair of distinct $z_i$. Separating a $z_i$
  from some $m\in\{x,y\}$ follows the same argument. To separate $m\in\{x,y\}$ from $z_i$ use $V_{m,i+1}$. $x$ is
  separated from $y$ by $V_{x,0}$, and $y$ is separated from $x$ by $V_{y,0}$.\\
$V$ is quasi-compact. To cover $x\in V$, any open cover needs to contain an open neighborhood containing some $V_{x,i}$,
  leaving only finitely many points in $V$ uncovered. A finite subcover is immediate.\\
Since all $z_i$ are isolated, any sticky sequence can only involve $x$ and $y$, and $(x,y)$ is indeed sticky. The implication in 
the B-space condition, which is only non-vacuous if more than two distinct points can form sequences from pairwise sticky 
sequences, is trivially satisfied.
\item 
Having described all open and quasi-compact open sets above, it is straightforward to see that all quasi-compact opens
are a topological basis. A more manageable quasi-compact open basis is given by
$\bigcup_{i=0}^{\infty}\{\{z_i\},V_{x,i},V_{y,i}\}$. Note that $V$ does not have a QCI basis.
\item The set of non-isolated points is $\{x,y\}$ and thus is finite.
\end{enumerate}
Therefore, Fischer's proposition asserts that there is a B\'{e}zout domain $R_V$ with $\mspec(R_V)\cong V$.

We look at the two sets $A=V\setminus\{z_0\}$ and $B=\{y\}$. By construction, $z_0$ is isolated, which means that its
complement $A$ is closed. $B=\{y\}=V\setminus\left(\left\{z_i: i\in\mathbb{N}\right\}\cup\{x\}\right)=V\setminus
V_{x,0}$ where $V_{x,0}$ is a basic open. Thus $B$ also is closed. The complement of $A$ is $V\setminus A = V\setminus
(V\setminus \{z_0\})=\{z_0\}$, an open singleton, hence clearly quasi-compact open. The complement of $B$ is
$V\setminus B = V\setminus \{y\}=V_{x,0}$. Any cover by basic open sets needs to cover $x\in V_{x,0}$. Hence the cover
needs to contain at least one neighborhood containing some $V_{x,n}$ that form the basic open neighborhoods of $x$. Take
such an $n$. But $V_{x,n}$ already covers cofinitely much of $V_{x,0}$, i.e., starting from that cover and using
$V_{x,n}$, we only need to choose at most one more open set of the cover for each of the $z_i$ for $i<n$. Hence we can
select a subcover of at most $1+n$ open sets, i.e. a finite cover. Thus $B$ also has quasi-compact complement. In
summary, $A$ and $B$ are both closed with quasi-compact complement.

By Lem.~\ref{lem:basicClosedSetsInBezoutSpectra}, there are ring elements $\tilde{a},\tilde{b}$ such that
$Z(\tilde{a})=A, Z(\tilde{b})=B$. Van den Dries and Macintyre~\cite[Sec. 2.10]{vDM90} imply that these $\tilde{a},\tilde{b}$
are not compatible with good factorization. Indeed, assume there were $\tilde{c}, \tilde{d} \in R_V$ such that $\tilde{a}
= \tilde{c}\tilde{d}, \gcd(\tilde{c},\tilde{b}) \sim 1$ and $\tilde{b}\in\sqrt{\left(\tilde{d}\right)}$. Topologically,
this means $Z(\tilde{c})\cup Z(\tilde{d})=Z(\tilde{a})=A=V\setminus\{z_0\}$, $\emptyset=Z(\tilde{c})\cap
Z(\tilde{b})=Z(\tilde{c})\cap B=Z(\tilde{c})\cap\{y\}$ and $\{y\}=B=Z(\tilde{b})\supseteq Z(\tilde{d})$. Taking
everything together, this implies $Z(\tilde{c})=V_{x,1}$ is a basic open set, which, however, is not closed (see the
description of the topology above) -- this is a contradiction, since a set of the form $Z(\tilde{c})$ by definition must
be closed. Thus $\tilde{a},\tilde{b}$ do not have a good factorization, finishing the example from \cite[Sec.
  2.10]{vDM90}.
\end{proof}

It is remarkable that one does not need any explicit algebraic knowledge of the ring $R_V$ to prove that it
does not have gf. In dialog with an LLM, I came up with the following approach in the same spirit of avoiding explicit
algebraic descriptions of $R_V$. Although this approach fails to demonstrate the failure of lcm splitting, it is
nevertheless instructive.

In Rem.~\ref{rmk:propertiesOfGCDDomains}, I mentioned that a domain $R$ is B\'{e}zout if and only if it is
a Pr\"{u}fer GCD domain. Since all localizations of Pr\"{u}fer domains are valuation rings, we get for
each maximal ideal $\mathfrak{m}\in\mspec(R)$ a valuation $v_{\mathfrak{m}}$ on $R$ by restriction. These
valuations behave well with respect to gcds and lcms. Explicitly, for $a,b\in R\setminus\{0\}$, we have
\[
v_{\mathfrak{m}}(\gcd(a,b))=
\min\left(v_{\mathfrak{m}}(a),v_{\mathfrak{m}}(b)\right) \text{ and }
v_{\mathfrak{m}}(\lcm(a,b))=
\max\left(v_{\mathfrak{m}}(a),v_{\mathfrak{m}}(b)\right).
\]
\begin{lemma}\label{lem:valuationsForLCMSplittings}
Let $R$ be a B\'{e}zout domain, let $a,b\in R\setminus\{0\}$, and let $c,d\in R$ satisfy lcm splitting for $a,b$,
i.e. $\lcm(a,b)\sim cd, c\mid a, d\mid b, \gcd(c,d)\sim1$. Then, for any $\mathfrak{m}\in\mspec(R)$, we have
\[
v_{\mathfrak{m}}(a) > v_{\mathfrak{m}}(b) \Rightarrow
(v_{\mathfrak{m}}(c) = v_{\mathfrak{m}}(a) \wedge
v_{\mathfrak{m}}(d) = 0) \quad\text{ and}\quad
v_{\mathfrak{m}}(a) < v_{\mathfrak{m}}(b) \Rightarrow
(v_{\mathfrak{m}}(c) = 0 \wedge
v_{\mathfrak{m}}(d) = v_{\mathfrak{m}}(b)).
\]
\end{lemma}
\begin{proof}
Wlog, we assume $v_{\mathfrak{m}}(a) > v_{\mathfrak{m}}(b)$; the second part follows by symmetry. Assumption $\gcd(c,d)\sim1$ implies 
\[
0=
v_{\mathfrak{m}}(1)=
v_{\mathfrak{m}}(\gcd(c,d))=
\min\left(v_{\mathfrak{m}}(c),v_{\mathfrak{m}}(d)\right)
\]
and $\lcm(a,b)\sim cd$ implies 
\[
0\leq
v_{\mathfrak{m}}(b)<
v_{\mathfrak{m}}(a)=
\max\left(v_{\mathfrak{m}}(a),v_{\mathfrak{m}}(b)\right)=
v_{\mathfrak{m}}(\lcm(a,b))=
v_{\mathfrak{m}}(cd)=
v_{\mathfrak{m}}(c)+v_{\mathfrak{m}}(d).
\]
From the first sequence of equalities, we have either $v_{\mathfrak{m}}(c)=0$ or $v_{\mathfrak{m}}(d)=0$ or both; 
the second inequality rules out the third option. If we had $v_{\mathfrak{m}}(c)=0$, then 
$v_{\mathfrak{m}}(d)=v_{\mathfrak{m}}(a)$, but $d\mid b$ implies 
$v_{\mathfrak{m}}(d)\leq v_{\mathfrak{m}}(b)<v_{\mathfrak{m}}(a)$, a contradiction. Thus the only remaining 
possibility is $v_{\mathfrak{m}}(d)=0$ and $v_{\mathfrak{m}}(c)=v_{\mathfrak{m}}(a)$. This proves the lemma.
\end{proof}

Now I come back to the ring $R_V$ with $\mspec(R_V)\cong V$ and $\tilde{a},\tilde{b}\in R_V$ as constructed in the
proof of Prop.~\ref{prop:spaceVFailureOfGF}. Since $\mspec(R_V)\cong V$, we will write the maximal ideals as $x, y,
z_i$. Clearly $\tilde{a},\tilde{b}$ are nonzero; otherwise their vanishing varieties would not be proper subsets of
$V$. Suppose that $\tilde{c}, \tilde{d}$ form an lcm splitting of $\tilde{a},\tilde{b}$. Recall that
the vanishing sets are $Z(\tilde{a})=A=\left\{z_i: i\in\mathbb{N}_{>0}\right\}\cup\{x,y\}$ and
$Z(\tilde{b})=B=\{y\}$. We need to consider three cases.
\begin{enumerate}
\item $\mathfrak{m}=z_0$. Here we have $v_{\mathfrak{m}}(\tilde{a})=0=v_{\mathfrak{m}}(\tilde{b})$. The lemma does not
  say anything about this case, but it is also irrelevant.
\item $\mathfrak{m}\in\left\{z_i: i\in\mathbb{N}_{>0}\right\}\cup\{x\}$. We have
  $v_{\mathfrak{m}}(\tilde{a})>0=v_{\mathfrak{m}}(\tilde{b})$, and the lemma implies $v_{\mathfrak{m}}(\tilde{c}) =
  v_{\mathfrak{m}}(\tilde{a})>0= v_{\mathfrak{m}}(\tilde{d})$.
\item $\mathfrak{m}=y$. We have $v_{\mathfrak{m}}(\tilde{a})>0<v_{\mathfrak{m}}(\tilde{b})$ and the lemma is not
  directly applicable. If we had $v_{\mathfrak{m}}(\tilde{a})\geq v_{\mathfrak{m}}(\tilde{b})>0$, this would imply that
  $\tilde{b}$ has strictly smaller valuations than $\tilde{a}$ at most places, with possible equality at
  $z_0$ and $y$. In any case, this would imply $\tilde{b}\mid \tilde{a}$, and an lcm splitting would trivially be given by
  $(\tilde{c},\tilde{d})=(\tilde{a},1)$. Therefore, in order to have a chance of demonstrating failure of lcm splitting,
  we would need $0<v_{\mathfrak{m}}(\tilde{a})<v_{\mathfrak{m}}(\tilde{b})$ at the only place where this is at all
  possible, i.e. $\mathfrak{m}=y$.
\end{enumerate}
By the abstract description used so far, we only know $v_{\mathfrak{m}}(\tilde{a})>0<v_{\mathfrak{m}}(\tilde{b})$ for
sure. If $0<v_{\mathfrak{m}}(\tilde{a})<v_{\mathfrak{m}}(\tilde{b})$, then we are en route to the failure of lcm
splitting. If $v_{\mathfrak{m}}(\tilde{a})\geq v_{\mathfrak{m}}(\tilde{b})>0$, a dialog with an LLM suggested taking a
high enough power of $\tilde{b}$, i.e. some $N\in\mathbb{N}$ with
$0<v_{\mathfrak{m}}(\tilde{a})<v_{\mathfrak{m}}(\tilde{b}^N)$, and using $\tilde{b}^N$ instead of $\tilde{b}$. That looked
very promising. However, for non-Noetherian domains $R$ like the ones we are interested in, in particular $R_V$, the
value groups need not be \textbf{Archimedean (i.e. embeddable in $\mathbb{R}$, also called rank-1 valuations)}. Instead,
we could have \textbf{higher-rank valuations} for which taking high enough powers is not guaranteed to succeed.

At this point, the abstract approach got stuck. In retrospect, it turned out that $R_V$, explicitly constructed following
Fischer's general approach, does have a valuation of rank $2$ at $\mathfrak{m}=y$ and that $\tilde{a},\tilde{b}$ as
described by van den Dries and Macintyre actually satisfy lcm splitting, as we will see later. The elegant abstract approach
succeeded in showing failure of gf, but to exhibit failure of lcm splitting, one has to work through the explicit
algebraic construction of $R_V$, which I will do now. I try to keep the general construction and the effect on our
example as separate as possible using $\bullet_V$ indication (and later $\bullet_M$) unless the notation becomes too
cluttered.

For a topological space $X$, denote the set of isolated points by $X^{\text{iso}}\subseteq X$ and set $X'=X\setminus
X^{\text{iso}}$. The assumptions of \cite[Prop. 3.3]{Fis79} include in particular that $X$ may have only finitely many
non-isolated points, i.e. that $N=|X'|<\infty$, and the construction proceeds by induction on $N$ with base cases being
$N\in\{0,1\}$. For $V$, we have $V'=V\setminus V^{\text{iso}}=\{x,y\}$, so we need to follow one induction step. The
induction distinguishes between all elements of the current exceptional induction set $X'=\{x_1,x_2,\ldots,x_N\}$
forming a sticky sequence, as defined in Def.~\ref{def:topPropertiesBSpace}, or not. For us, $(x,y)$ forms a sticky
sequence, as any open neighborhoods of $x$ resp. $y$ intersect, since the basic open
neighborhoods intersect nontrivially: $V_{x,m}\cap V_{y,n}=\{z_i:i\geq\max(m,n)\}\neq\emptyset$.

The construction then lets one choose quasi-compact open neighborhoods for each $x_i\in X'$. I take them as
large as possible: $\alpha_x=V_{x,0}, \alpha_y=V_{y,0}$, which yields a subset $X\supseteq\Omega$, in our case
$\Omega_V=\left\{z_i: i\in\mathbb{N}\right\}$. The complement $X\setminus\Omega$, which in our case is $\{x,y\}$ and
coincides by coincidence with $V'$ (in general only $X'\subseteq X\setminus\Omega$ is guaranteed), then loses
stickiness (fewer pairs in $X'$ are sticky in the subspace topology of $X\setminus\Omega$ than in the
topology of $X$ itself), and the other case of the induction is employed to incorporate the set $X\setminus\Omega$ into a
suitable spectral tree $Y_1$. In our case, $Y_1=V\setminus\Omega_V=\{x,y\}$ already works without having to add anything
additionally (recall that forests with several connected components are possible for spectral trees). Now, back in the
sticky branch of the induction, we add points $\Gamma=\{p_x:x\in X\setminus X'\}\dot{\cup}\{p\}$, which will be
generizations of the established points of $X$. In our case, we write $p_i$ for $p_{z_i}$ for simplicity. For the
suitably extended topology, we will then have $z_i>_sp_i$ for all $i\in\mathbb{N}$ as well as $x>_sp<_sy$ (recall the
specialization order from Def.~\ref{def:topPropertiesBSpace}). Fischer also introduces the notation
$\Theta=\Omega\cup\Gamma$. The complete set for the induction step is given by $Y=\Theta\cup Y_1$.

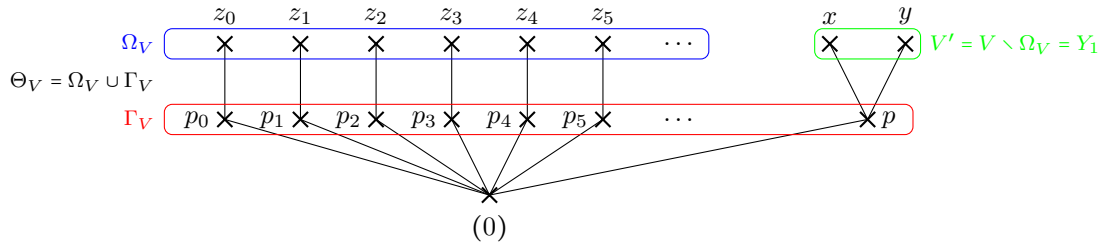
\begin{figure}[!h]
\centering
\begin{tikzpicture}[croix/.pic={\draw[thick,black] (-0.1,-0.1) -- (0.1,0.1) (-0.1,0.1)--(0.1,-0.1);}]
\foreach \z in {0,1,...,5}{
\draw (\z,0) pic {croix} node[above=4]{$z_{\z}$};
\draw (\z,-1) pic {croix} node[left=2]{$p_{\z}$};
\draw (\z,0) -- (\z,-1) -- (3.5,-2);
}
\draw (6,0) node{$\ldots$};
\draw (6,-1) node{$\ldots$};
\draw (8,0) pic {croix} node[above=4]{$x$};
\draw (9,0) pic {croix} node[above=4]{$y$};
\draw (8.5,-1) pic {croix} node[right=2]{$p$};
\draw (3.5,-2) pic {croix} node[below=4]{$(0)$};
\draw (8,0) -- (8.5,-1) -- (9,0) (8.5,-1) -- (3.5,-2);
\draw [color=blue, rounded corners=3] (6.4,0.2) rectangle (-0.8,-0.2) node[left,yshift=2mm] {\footnotesize$\Omega_V$};
\draw [color=green, rounded corners=3] (7.8,0.2) rectangle (9.2,-0.2) node[right,yshift=2mm] {\footnotesize$V'=V\setminus\Omega_V=Y_1$};
\draw [color=red, rounded corners=3] (9.1,-0.8) rectangle (-0.8,-1.2) node[left,yshift=2mm] {\footnotesize$\Gamma_V$};
\draw (-0.8,-0.5) node[left] {\footnotesize$\Theta_V=\Omega_V\cup\Gamma_V$};
\end{tikzpicture}
\caption{\footnotesize The spectral tree $M'=M\cup\{(0)\}$ with various named sets and specialization order indicated.}
\label{fig:MSpecTop}
\end{figure}

Fischer's construction would in general proceed inductively like this, but since our exceptional set at the start
$V'=\{x,y\}$ had only two elements, the construction already stops here. At the end, an additional generic point $(0)$
is added, which is the generization of all points constructed so far -- this is described in the proof of
\cite[Thm. 2.1]{Fis78}. This makes sure that the spectral space belongs to an integral domain and not an arbitrary
ring. I summarize the resulting space $M'=M\cup\{(0)\}$ in the following definition, describe the topology and
illustrate the space in Fig.~\ref{fig:MSpecTop}.

\begin{definition}\label{def:spaceM}
The space $M=\left\{z_i,p_i: i\in\mathbb{N}\right\}\cup\{x,y,p\}\supset V$ is endowed with the topology induced by the
following open basis $\mathcal{B}$ (the numbering follows the general construction in \cite[p. 252]{Fis79} and seems to
be off at iii. for this reason):
\begin{enumerate}[i.]
\item $\{p_i\}$ for each $i\in\mathbb{N}$,
\item $\{p_i,z_i\}$ for each $i\in\mathbb{N}$,
\item no further open sets of this type,
\item $\{p,w\}\cup\{p_i:i\geq m\}\cup\{z_i:i\geq n\}$ for each pair of integers $0\leq m\leq n$ and $w\in\{x,y\}$,
\item $\{p\}\cup\{p_i:i\geq m\}\cup\{z_i:i\geq n\}$ for each pair of integers $0\leq m\leq n$.
\end{enumerate}
Set $M'=M\cup\{(0)\}$ with open basis $\mathcal{B}'=\{U\cup\{(0)\}: U\in\mathcal{B}\}$.
\end{definition}

It will be useful to describe several aspects of the topology explicitly. A subset $\{z_i: i\in Z\}\cup\{p_i: i\in P\}\cup
L\subseteq M$ for $Z,P\subseteq\mathbb{N}, L\subseteq\{x,y,p\}$ in the \textbf{spectral topology on $M$} is
\begin{itemize}
\item \textbf{closed} iff ($L=\{x,y,p\}$, $Z\supseteq P$ and otherwise arbitrary) or ($p\not\in L$ and $Z\supseteq P$ finite),
\item \textbf{open} iff ($L=\emptyset$, $Z\subseteq P$ and otherwise arbitrary) or ($p\in L$ and $Z\subseteq P$ cofinite),
\item \textbf{quasi-compact open} iff ($L=\emptyset$ and $Z\subseteq P$ finite) or ($p\in L$ and $Z\subseteq P$ cofinite).
\end{itemize}
In particular, all $\{z_i\}$ are closed, as are $\{x\},\{y\}$, and all $\{p_i\}$ are
quasi-compact open (type i.\ in Def.~\ref{def:spaceM}). Therefore, these singletons are closed in the patch
topology. Similarly, their complements are also closed in the patch topology. Thus every point of $M$ except $p$ is isolated in
the patch topology.
A subset $\{z_i: i\in Z\}\cup\{p_i: i\in P\}\cup
L\subseteq M$ for $Z,P\subseteq\mathbb{N}, L\subseteq\{x,y,p\}$ in the \textbf{patch topology on $M$} is
\begin{itemize}
\item \textbf{closed} iff $p\in L$ or ($p\not\in L$ and $Z, P$ finite),
\item \textbf{open} iff $p\not\in L$ or ($p\in L$ and $Z, P$ cofinite),
\item \textbf{quasi-compact open} iff ($p\not\in L$ and $Z, P$ finite) or ($p\in L$ and $Z, P$ cofinite).
\end{itemize}
The \textbf{constructible sets for the spectral topology on $M$} are the clopen sets for the patch topology. Thus a set
in the above presentation is \textbf{constructible} iff $p\in L$ and $Z, P$ cofinite or $p\not\in L$ and $Z, P$
finite. In this special case, for $X=M$, a set is constructible in the spectral topology on $M$ iff it is quasi-compact
open in the patch topology on $M$.

\begin{remark}\label{rmk:propertiesOfM}
The proof of \cite[Prop. 3.3]{Fis79} guarantees that $M'\cap V\cong V$ are homeomorphic as topological spaces. The reader is
invited to check this directly from the explicit descriptions of the spectral topologies on $M'$ resp. $M$ and $V$
given above.
\end{remark}

Next I construct from this spectral space an associated lattice ordered group. I repeat
Fischer~\cite[Thm. 2.1]{Fis78} as the key result around this construction:
\begin{prop}\label{prop:Fischer78_2.1}
Let $X$ be a spectral tree such that the closure (in the spectral topology) of every constructible
set is constructible. Then $X$ is homeomorphic to the prime spectrum of a B\'{e}zout ring $R$, i.e. $\spec(R)\cong X$.
\end{prop}
Prop.~\ref{prop:Fischer78_2.1} does not assert that $R$ is a domain. Fischer comments in \cite[p. 1727]{Fis78} that the resulting ring is a domain
if $X$ has a unique minimal element, which is the case in our example.

Using the description of the topology above, we can check the conditions of
Prop.~\ref{prop:Fischer78_2.1}. Fig.~\ref{fig:MSpecTop} suggests that $M$ is a spectral tree, but one can also check
this formally by determining the closures of all singletons with the help of the description of closed sets
above. E.g., any closed set containing $p$ must also contain $\{x,y\}$. Thus $x,y$ are specializations of $p$, and since
$\{p,x,y\}$ is closed, they are also the only specializations, etc.

To see that -- all in the spectral topology -- the closure of any constructible set $C\subseteq M$ is constructible,
we distinguish two cases. In the first case, $p\not\in C$, we have $C=\{z_i:i\in Z\}\cup\{p_i:i\in P\}\cup L$ with
$L\subseteq\{x,y\}$, and $C$ is finite. The closure is $\overline{C}^{\text{cl}}=\{z_i:i\in Z\cup P\}\cup\{p_i:i\in
P\}\cup L$, again finite, and easily seen to be constructible. If $p\in C$, then $C=\{z_i:i\in Z\}\cup\{p_i:i\in P\}\cup
L$ with $L\subseteq\{p,x,y\}$ such that $Z,P$ are cofinite. The closure is $\overline{C}^{\text{cl}}=\{z_i:i\in
Z\cup P\}\cup\{p_i:i\in P\}\cup \{p,x,y\}$, again with $Z\cup P,P$ cofinite, which is clearly constructible by
the list above.

According to Fischer's proof~\cite[Thm. 2.1]{Fis78}, the lattice ordered group
$G=C(X,\mathbb{Z})=C(X^{\text{patch}},\mathbb{Z}^{\text{disc}})$ consists of integer-valued functions that are
continuous for the patch topology on $X$ and the discrete topology on $\mathbb{Z}$. For $f\in G$, denote by
$\supp(f)\subseteq X$ the support of $f$ defined by $\supp(f)=\{x\in X:f(x)\neq 0\}$, and by $\msupp(f)\subseteq
\supp(f)\subseteq X$ the minimal support of $f$ (with respect to the specialization order) defined by $\msupp(f)=\{x\in
X:f(x)\neq 0\wedge \forall y\in X: y<_sx \Rightarrow f(y)=0\}$. The \textbf{positive cone} is $G^+=\left\{f\in G:
\forall x\in\msupp(f): f(x)>0\right\}$. Note that since $\msupp(0)\subseteq\supp(0)=\emptyset$, the positivity condition
is vacuous for the $0$ function and thus $0\in G^+$. The \textbf{lattice order}, which is a \textbf{partial order} in
general (i.e. some pairs of elements may be incomparable), is defined via the positive cone: $f\geq_G g\Leftrightarrow
f-g\in G^+$.

In our case, $X=M$, we describe $G_M$ and $G_M^+$ explicitly. For $f\in G_M$, consider $f(p)\in\mathbb{Z}$. In the
discrete topology, $\{f(p)\}$ is open. Continuity forces $U_p=f^{-1}\left(\{f(p)\}\right)\subseteq M$ to be
open (in the patch topology), and the only way this can happen according to the description of the topology above is with
$|M\setminus U_p|<\infty$, i.e. if $U_p$ is cofinite. So, except for finitely many (but otherwise arbitrary) points from
$\{z_i,p_i: i\in\mathbb{N}\}\cup\{x,y\}$, $f$ takes the value $f(p)$. This also implies that the image of $f$ is
finite: $|f(M)|<\infty$.

Note that $f\in G^+$ need not take only nonnegative values. $f(x)<0$ is allowed, as long as the minimal generization
$y<_sx$ with $f(y)\neq0$ satisfies $f(y)>0$. In this case $\msupp(f)\subsetneq\supp(f)$, i.e., the minimal support is a proper subset of the
support. This also has consequences for the lattice operations $\inf,\sup$, which correspond directly to $\gcd,\lcm$ in
the constructed B\'{e}zout domain. For $f,g\in G$ and $x\in X$, we set 
\[
\inf(f,g)(x)=\begin{cases}f(x), &\text{ if }G(x)\cap\msupp(f-g)=\emptyset\\
f(x), &\text{ if }\{y\}=G(x)\cap\msupp(f-g)\neq\emptyset\wedge f(y)< g(y)\\ 
g(x), &\text{ if }\{y\}=G(x)\cap\msupp(f-g)\neq\emptyset\wedge f(y)> g(y).\end{cases}
\]
Note that since each $G(x)$ is a totally ordered set ($X$ is assumed to be a spectral tree), the intersection with the
minimal support $G(x)\cap\msupp(f-g)=G_{x,f,g}$ is either a singleton or empty. We can describe $G_{x,f,g}$ as the
minimal generization $y\leq_s x\in X$ of $x$ where $f$ and $g$ differ, if they differ at all. If they do not differ (on
the generizations of $x$ -- outside these generizations, $f$ and $g$ may still differ), we take the only meaningful value
$f(x)=g(x)$. If they differ, the smaller value on the minimal differing $y$ determines which value to take, i.e., we take $f(x)$ if
$f(y)<g(y)$, and $g(x)$ if $f(y)>g(y)$. If the space $X$ has proper generizations, this inevitably
leads to higher-rank valuations. This is actually the case for $X=M$, giving a posteriori insight into the failure of the
initial strategy for a counterexample. Similarly,
\[
\sup(f,g)(x)=\begin{cases}f(x), &\text{ if }G(x)\cap\msupp(f-g)=\emptyset\\
f(x), &\text{ if }\{y\}=G(x)\cap\msupp(f-g)\neq\emptyset\wedge f(y)> g(y)\\ 
g(x), &\text{ if }\{y\}=G(x)\cap\msupp(f-g)\neq\emptyset\wedge f(y)< g(y).\end{cases}
\]

I summarize these findings in the following definition for $M$.
\begin{definition}\label{def:groupG}
The lattice ordered group $G_M\subset\text{Map}(M,\mathbb{Z})$ consists of those functions $f:M\rightarrow\mathbb{Z}$ for
which all but finitely many values agree with $f(p)$, i.e. for which $|f^{-1}(\mathbb{Z}\setminus\{f(p)\})|<\infty$.\\ 
The positive cone $G_M^+$ consists of those $f\in G_M$ such that, for all $i\in\mathbb{N}$, we have $f(p_i)>0 \vee
(f(p_i)=0\wedge f(z_i)\geq 0)$ and additionally $f(p)>0 \vee (f(p)=0\wedge f(x)\geq 0\leq f(y))$. This defines the
(partial) lattice order via $f\geq_{G_M}g\Leftrightarrow f-g\in G_M^+$.\\
We also have, for $f,g\in G_M$ and $m\in M$,
\begin{gather*}
\inf(f,g)(m)=\begin{cases}
\min(f(m),g(m)) &\text{ if } m\in\{p_i:i\in\mathbb{N}\}\cup\{p\}\\
f(m) &\text{ if } m=z_i\wedge f(p_i)< g(p_i)\\
\min(f(m),g(m)) &\text{ if } m=z_i\wedge f(p_i)= g(p_i)\\
g(m) &\text{ if } m=z_i\wedge f(p_i)> g(p_i)\\
f(m) &\text{ if } m\in\{x,y\}\wedge f(p)< g(p)\\
\min(f(m),g(m)) &\text{ if } m\in\{x,y\}\wedge f(p)= g(p)\\
g(m) &\text{ if } m\in\{x,y\}\wedge f(p)> g(p)
\end{cases}\\
\sup(f,g)(m)=\begin{cases}
\max(f(m),g(m)) &\text{ if } m\in\{p_i:i\in\mathbb{N}\}\cup\{p\}\\
f(m) &\text{ if } m=z_i\wedge f(p_i)> g(p_i)\\
\max(f(m),g(m)) &\text{ if } m=z_i\wedge f(p_i)= g(p_i)\\
g(m) &\text{ if } m=z_i\wedge f(p_i)< g(p_i)\\
f(m) &\text{ if } m\in\{x,y\}\wedge f(p)> g(p)\\
\max(f(m),g(m)) &\text{ if } m\in\{x,y\}\wedge f(p)= g(p)\\
g(m) &\text{ if } m\in\{x,y\}\wedge f(p)< g(p)
\end{cases}
\end{gather*}
\end{definition}

The final step is the construction of the actual B\'{e}zout domain $R$ from a lattice ordered group $G$. Here I follow
Gilmer~\cite[Thm. 18.6]{Gil92}. Also note that the notions of \textbf{semi-/demivaluation resp. semi-/demivaluation
  ring} employed at the end of its proof can be found in \cite[Ch. 16]{Gil92}. We start with any field $k$, for
concreteness, $k=\mathbb{Q}$, and form the group ring $k[G]$. Associate to each $g\in G$ a formal symbol $x^g$. Then
elements of $k[G]$ have a canonical form $\sum_{i=1}^na_{g_i}x^{g_i}$, unique up to the order of summation, with
$a_{g_i}\in k\setminus\{0\}$ and $(g_i)_{i=1}^n\in G^n$ a finite sequence of pairwise distinct elements of $G$ for some
integer $n\in\mathbb{N}$. Note that for $n=0$, we get the zero $0=0_{k[G]}\in k[G]$, an empty sum. For $g_1=0:
X\rightarrow\mathbb{Z},x\mapsto 0$ and $a_{g_1}=1\in k$, we obtain the unit $1=1_{k[G]}=1\cdot x^{0}\in k[G]$ as a one-element
sum. For canonical $r=\sum_{i=1}^na_{g_i}x^{g_i}\in k[G]\setminus\{0\}$, we set its value $w(r)=\inf((g_i)_{i=1}^n)$ --
recall that, similar to Rem.~\ref{rmk:multipleGCD}, multiple infima in a lattice ordered group can be well-defined through
iteration in the same way as multiple gcds in a GCD domain.

After establishing the relevant properties, one forms the fraction field $\Frac(k[G])$ and extends $w$ as
follows: any nonzero element of $\Frac(k[G])$ can be written (non-uniquely) as a quotient of canonical polynomials,
i.e. $r=\frac{\sum_{i=1}^ma_{f_i}x^{f_i}}{\sum_{i=1}^nb_{g_i}x^{g_i}}$, and
$w(r)=w\left(\sum_{i=1}^ma_{f_i}x^{f_i}\right)-w\left(\sum_{i=1}^nb_{g_i}x^{g_i}\right)=\inf((f_i)_{i=1}^m)-\inf((g_i)_{i=1}^n)\in
G$ yields a well-defined demivaluation. Its demivaluation ring $R=\Frac(k[G])_w$ consists of those
$r\in\Frac(k[G])\setminus\{0\}$ with $w(r)\geq_G 0$, i.e. $w(r)\in G^+$, together with $0$ (one could set
$w(0)=\infty$). The proof also shows that it is always possible to write a nonzero ring element $r\in R\setminus\{0\}$
in the form $r=\frac{1_k\cdot x^{w(r)}}{1_{k[G]}}\cdot r'$ for some unique $r'\in R$ with $w(r')=0$, i.e. with $r'$ a
unit. We shorten the notation to $r=x^{w(r)}\cdot r'$. For $G=G_M$, this finally realizes the B\'{e}zout domain $R_V$
explicitly in which the counterexample to gf abstractly took place. To summarize:

\begin{definition}\label{def:bezoutDomainR}
The ring $R_V\subset\Frac(\mathbb{Q}[G_M])$ consists of quotients (with nonzero denominator) of canonical group ring
polynomials $r=\frac{\sum_{i=1}^ma_{f_i}x^{f_i}}{\sum_{i=1}^nb_{g_i}x^{g_i}}$ with $m,n\in\mathbb{N}, f_i,g_i\in
G_M=C(M,\mathbb{Z})$, and $a_{f_i},b_{g_i}\in\mathbb{Q}\setminus\{0\}$, such that its demivaluation $w(r)\geq_{G_M}0$ is
nonnegative, i.e. such that $\inf((f_i)_{i=1}^m)-\inf((g_i)_{i=1}^n)\in G_M^+$, and additionally a zero $0_{R_V}\in R_V$.

The ring operations are the usual ones for such quotients. The theorems used in the construction guarantee that $R_V$ is
a B\'{e}zout domain with $\spec(R_V)\cong\spec(G_M)\cup\{(0)\}\cong M'$ and $\mspec(R_V)\cong V$. Any nonzero ring
element $r\in R_V\setminus\{0\}$ can be written as $r=x^{w(r)}\cdot r'$ for a unique unit $r'\in R_V^{\times}$ with
$w(r')=0$. While $r'$ is unique, its presentation as a quotient of canonical group ring polynomials is not.
\end{definition}

Before presenting the counterexample to lcm splitting, I revisit the counterexample to gf from
Prop.~\ref{prop:spaceVFailureOfGF}. Recall that we showed the existence of $\tilde{a},\tilde{b}\in R_V\setminus\{0\}$
with $Z(\tilde{a})=A=V\setminus\{z_0\}, Z(\tilde{b})=B=\{y\}$, where $Z$ denotes the vanishing subspace in the maximal
spectrum. Since $Z(\tilde{a})$ is infinite, our description of closed sets in the spectral topology of $M$ forces the
vanishing subspace in the prime spectrum to contain $p$, i.e. $p\in V(\tilde{a})$ or $w(\tilde{a})(p)>0$, while
$Z(\tilde{b})$ finite forces $p\not\in V(\tilde{b})$ resp. $w(\tilde{b})(p)=0$. Since $y\in\overline{\{p\}}^{\text{cl}}$,
no power of $\tilde{b}$ can ever surpass $\tilde{a}$ in the rank-2 valuation at $y$,
i.e. $v_y(\tilde{a})\nless v_y(\tilde{b}^{N})$ for all $N\in\mathbb{N}$. Thus, counterexamples to gf of the type in
\cite[Sec. 2.10]{vDM90} will never yield a counterexample to lcm splitting.

Explicitly, take $g_A, g_B:M\rightarrow\mathbb{Z}$ such that $g_A(m)=1$ for all $m\in M\setminus\{z_0,p_0,y\}$,
$g_B(m)=0$ for all $m\in M\setminus\{y\}$, and $g_A(z_0)=g_A(p_0)=g_A(y)=0$, while $g_B(y)=1$. Set
$\tilde{a}=\frac{1\cdot x^{g_A}}{1}=x^{g_A}, \tilde{b}=\frac{1\cdot x^{g_B}}{1}=x^{g_B}\in R_V\setminus\{0\}$. Then
$\tilde{b}\mid \tilde{a}$. Thus, by (2), we have $\lcm(\tilde{a},\tilde{b})\sim \tilde{a}, \gcd(\tilde{a},\tilde{b})\sim
\tilde{b}$. Hence $c=\tilde{a},d=1_{R_V}$ are easily checked to witness lcm splitting of $(\tilde{a},\tilde{b})$.

Note that the topological description of Prop.~\ref{prop:spaceVFailureOfGF} 
neither fixes the ring $R_V$ (e.g. replace $k=\mathbb{Q}$ by another field)
nor $(\tilde{a},\tilde{b})$ uniquely, e.g. $(4\tilde{a}^5,-\frac{6}{7}\tilde{b}^8)$ 
would be equally valid as 
a counterexample to gf. I will use the choice just described in the following when referring to 
Prop.~\ref{prop:spaceVFailureOfGF}.

\begin{theorem}\label{thm:counterexampleToLCMSplitting}
For $G_M$ as in Def.~\ref{def:groupG}, take $g_a, g_b:M\rightarrow\mathbb{Z}$ such that $g_a(m)=1=g_b(m)$ for all $m\in
M\setminus\{x,y\}$ and $g_a(x)=0=g_b(y), g_a(y)=-1=g_b(x)$. With the notation of Def.~\ref{def:bezoutDomainR}, set
$a=x^{g_a}, b=x^{g_b}\in R_V\setminus\{0\}$. Then $a,b$ admit no lcm splitting, and lcm splitting fails in $R_V$.
\end{theorem}

\begin{proof}
We first verify that all quantities are well-defined in the sets we asserted above. Since $g_a,g_b$ take the value $1$
outside the finite set $\{x,y\}$, in particular at $p$, these functions belong to $G_M$. They even belong to
$G_M^+$: they are nonnegative except for $g_a(y)=-1=g_b(x)$, and since $x>_{s}p<_{s}y$ and
$g_a(p)=g_b(p)=1>0$, this is compatible with being in $G_M^+$. In particular, $w(a)=g_a-0=g_a, w(b)=g_b-0=g_b\in G_M^+$.
Thus $a,b\in R_V\subseteq\Frac(\mathbb{Q}[G_M])$ are nonzero ring elements.

Applying the description of the supremum in Def.~\ref{def:groupG}, we see that $w(\lcm(a,b))=\sup(w(a),w(b))=g_l$ with
$g_l:M\rightarrow\mathbb{Z}$ where $g_l(m)=1$ for all $m\in M\setminus\{x,y\}$ and $g_l(x)=0=g_l(y)$. This lcm can be
represented by $l=x^{g_l}$ (uniquely up to units). 

Now suppose that $(c,d)\in R_V^2$ is an lcm splitting of $(a,b)$. Up to units, we can represent them by monomials
$c=x^{g_c},d=x^{g_d}$, and by the properties of $G_M^+$, we have $g_c(p)\geq 0\leq g_d(p)$. Since $cd\sim\lcm(a,b)=l$ and
multiplication in $R_V$ corresponds to addition in $G_M$, this implies $g_c(p)+g_d(p)=1$. By symmetry, we may assume
$g_c(p)=1,g_d(p)=0$. Since $c\mid a$ and $g_c(p)=1=g_a(p)$, this forces $g_c(y)\leq g_a(y)=-1$. Since $cd\sim l$, we
have at $y$ that $-1+g_d(y)\geq g_c(y)+g_d(y)=g_l(y)=0$, i.e. $g_d(y)\geq 1$. We also need to satisfy the splitting condition
$\gcd(c,d)\sim1$, i.e. $\inf(g_c,g_d)(m)=0$ for all $m\in M$. In particular, for $m=y$ we need
$\inf(g_c,g_d)(y)=0$. Since $y>_s p$, we first check at $p$ and have $g_c(p)=1>0=g_d(p)$. Thus
$\inf(g_c,g_d)(y)=g_d(y)\geq 1$, which implies $\inf(g_c,g_d)(y)\neq 0$, a contradiction.
\end{proof}

I summarize some further properties of $R_V$.
\begin{cor}\label{cor:SimplePropertiesOfR}
$R_V$ is a B\'{e}zout domain that is neither gf, nor adequate, nor lcm splitting, but it is an elementary divisor domain and K-Hermite.
\end{cor}

\begin{proof}
By construction (cf. Def.~\ref{def:bezoutDomainR}), $R_V$ is a B\'{e}zout domain. By
Thm.~\ref{thm:counterexampleToLCMSplitting}, it is not lcm splitting.\\
By Thm.~\ref{thm:adequateImpliesLCMSplitting}, a B\'{e}zout domain such as $R_V$ that is not lcm splitting cannot be
adequate.\\
By ``gf $\Rightarrow$ adequate'' (cf. Rem.~\ref{rmk:propertiesOfGCDDomains}), $R_V$ cannot be gf either.\\
By \cite[Thm. 3.1.2]{RRW24}, a B\'{e}zout domain with countably many maximal ideals is an elementary divisor domain. Clearly,
$V=\{z_i:i\in\mathbb{N}\}\cup\{x,y\}$ is countable. Thus $R_V$ is an elementary divisor domain.\\
By ``elementary divisor $\Rightarrow$ K-Hermite'' (cf. Rem.~\ref{rmk:propertiesOfGCDDomains}), $R_V$ is
K-Hermite. Additionally, $R_V$ is a B\'{e}zout domain, which is equivalent to K-Hermite for domains
(cf. Rem.~\ref{rmk:propertiesOfGCDDomains}).
\end{proof}

Next, I examine some properties of pairs of elements of $R_V$.
\begin{prop}\label{prop:RIsNotAdequate}
The counterexample to gf from Prop.~\ref{prop:spaceVFailureOfGF} is also a counterexample to adequate. Let $g_A,g_B\in
G_M$ be given by $g_A(m)=1$ for all $m\in M\setminus\{z_0,p_0,y\}$ and $g_A(z_0)=g_A(p_0)=g_A(y)=0$, and $g_B(m)=0$ for all $m\in
M\setminus\{y\}$ and $g_B(y)=1$. Then $\tilde{a}=x^{g_A}, \tilde{b}=x^{g_B}$ form a counterexample.\\
However, in reversed order, $(\tilde{b},\tilde{a})$ satisfies gf and adequate.
\end{prop}

\begin{proof}
Assume we had a pair $(c,d)$ satisfying the adequate conditions, wlog of the form $c=x^{w(c)}=x^{g_C},
d=x^{w(d)}=x^{g_D}$. Necessarily $g_C(p)=0$, because $g_C(p)\geq 1$ would force $\inf(g_C,g_B)(y)=g_B(y)=1>0$,
contradicting $\gcd(c,\tilde{b})\sim1=x^{0}$. Since all but finitely many values of $g_C$ must agree with $g_C(p)=0$ (by
continuity of $g_C$), there is some $i\neq 0$ with $g_C(z_i)=g_C(p_i)=0$. The condition $\tilde{a}=cd$ implies
$1=g_A(z_i)=g_C(z_i)+g_D(z_i)=0+g_D(z_i)$, i.e. $g_D(z_i)=1$. Now define $g_{D'}$ by $g_{D'}(m)=0$ for all $m\in M\setminus\{z_i\}$
and $g_{D'}(z_i)=1$. By construction, $0_{G_M}<_{G_M}g_{D'}\leq_{G_M}g_D$. Thus $d'=x^{g_{D'}}\in R_V\setminus R_V^{\times}$
is a non-unit with $d'\mid d$. On the other hand, $g_{D'}(m)=0$ for all $m\neq z_i$ forces $\inf(g_{D'},g_B)(m)=0$ at the
same places $m$. Since also $g_B(z_i)=0=g_B(p_i)=g_{D'}(p_i)$, we get $\inf(g_{D'},g_B)(z_i)=0$; hence
$\gcd(\tilde{b},d')\sim1$. This contradicts the adequate property and establishes the claim.

To prove that $(\tilde{b},\tilde{a})$ satisfies gf, set $c=1, d=\tilde{b}$. Clearly $\tilde{b}=cd$ and
$\gcd(c,\tilde{a})=\gcd(1,\tilde{a})\stackrel{\text{\tiny (6), (9)}}{\sim} 1$. Let $g_{D''}\in G_M$ be given by
$g_{D''}(m)=1$ for all $m\in M\setminus\{z_0, p_0,y\}$, $g_{D''}(z_0)=g_{D''}(p_0)=0$ and $g_{D''}(y)=-1$. One checks
that $g_{D''}\in G_M^+$ and $d''=x^{g_{D''}}\in R_V$ is well-defined. By construction,
$\tilde{a}=d''\tilde{b}=d''d\in\sqrt{(d)}$, which shows the gf property. Lem.~\ref{lem:gfImpliesAdequateForPairs} also implies
the adequate property.
\end{proof}

\begin{remark}\label{rmk:propertiesOfPairsInCounterexampleRing}
Combining the results from Prop.~\ref{prop:spaceVFailureOfGF} and Prop.~\ref{prop:RIsNotAdequate} with the comments
immediately before Thm.~\ref{thm:counterexampleToLCMSplitting}, we see that the B\'{e}zout domain $R_V$ has a pair of elements
$(\tilde{a},\tilde{b})$ that violates the gf and adequate conditions, but satisfies lcm splitting. In reverse order,
$(\tilde{b},\tilde{a})$ satisfies all three properties.

Thm.~\ref{thm:counterexampleToLCMSplitting} gives a pair $(a,b)$ in the same $R_V$ that violates lcm splitting. Surprisingly,
this pair satisfies gf. Set $(c,d)=(1,a)$. Then $a=cd, \gcd(c,b)=\gcd(1,b)\stackrel{\text{\tiny (6),
    (9)}}{\sim}1$ are clearly satisfied. To see $b\in\sqrt{(d)}=\sqrt{(a)}$, it suffices to find some $f\in R_V$ with
$b^2=fa$. Set $f=x^{g_f}$ with $g_f:M\rightarrow\mathbb{Z}$ given by $g_f(m)=1$ for all $m\in M\setminus\{x,y\}$ and
$g_f(x)=-2, g_f(y)=1$. It is straightforward to check $g_f\in G_M^+$ and $b^2=fa$ as
claimed. Lem.~\ref{lem:gfImpliesAdequateForPairs} then shows that $(a,b)$ satisfies the conditions of adequate. By
the symmetry of the construction of $a$ and $b$, the reversed pair $(b,a)$ is also gf and adequate.

It is surprising that the pair $(a,b)$ failing lcm splitting satisfies the adequate property, while adequate (of all
nonzero pairs in a ring) implies lcm splitting by Thm.~\ref{thm:adequateImpliesLCMSplitting}. However, the proof of
that theorem uses the adequate properties of pairs potentially different from $(a,b)$, namely those built from $g=\gcd(a,b), \frac{a}{g},
\frac{b}{g}$, etc. This is what happens here: while $g$ is quite ``large'' compared to $a$ and $b$, the fractions
$\frac{a}{g}, \frac{b}{g}$ are ``multiplicatively small'', almost $1_{R_V}$. In fact, $\frac{b}{g}=\tilde{b}$ is the element from the
other pair. These small elements violate the adequate conditions needed in the proof of
Thm.~\ref{thm:adequateImpliesLCMSplitting}.
\end{remark}

\subsection{Second example}
\label{sec:counterexampleLCMSplitting2}

I found the second example after finishing the first. It is arguably simpler and gives a counterexample pair $(a,b)$
with properties similar to those in Sec.~\ref{sec:counterexampleLCMSplitting1}. The previous construction might,
however, be more flexible and give rise to modified examples with new properties, cf. Sec.~\ref{sec:final} for some
concrete ideas. I use the concepts and terminology developed for the first example for easier comparison and
presentation.

\newpage
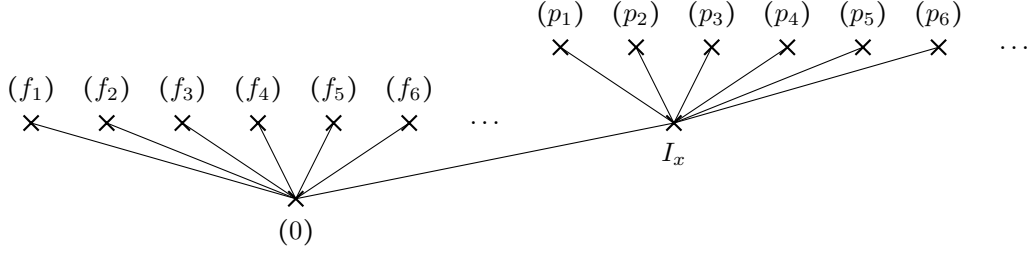
\begin{figure}[!h]
\centering
\begin{tikzpicture}[croix/.pic={\draw[thick,black] (-0.1,-0.1) -- (0.1,0.1) (-0.1,0.1)--(0.1,-0.1);}]
\foreach \z in {1,2,...,6}{
\draw ({\z-1},-1) pic {croix} node[above=4]{$(f_{\z})$};
\draw ({\z-1},-1) -- (3.5,-2);
}
\draw (13,0) node{$\ldots$};
\draw (6,-1) node{$\ldots$};
\draw (8.5,-1) pic {croix} node[below=4]{$I_x$};
\draw (3.5,-2) pic {croix} node[below=4]{$(0)$};
\foreach \z in {1,2,...,6}{
\draw ({\z+6},0) pic {croix} node[above=4]{$(p_{\z})$};
\draw ({\z+6},0) -- (8.5,-1);
}
\draw (8.5,-1) -- (3.5,-2);
\end{tikzpicture}
\caption{\footnotesize The spectral tree $\spec(S)$ for $S=\mathbb{Z}+x\mathbb{Q}[x]$ with indicated specialization
  order.}
\label{fig:SSpecTop}
\end{figure}

Let $S=\mathbb{Z}+x\mathbb{Q}[x]\subset\mathbb{Q}[x]$. By \cite[Thm. 4.6 (6)]{And00}, it is a B\'{e}zout domain. By
Roitman~\cite[Ex. 3.3]{Roi13}, $S$ is even an elementary divisor domain. Properties of this ring can be found in
Hutchins~\cite[Ex. 32, p. 67]{Hut81}. According to \cite[Thm.  1.13]{Smi92}, irreducible elements of $S$ are prime and
they are classified (up to units) as follows:
\begin{enumerate}[(a)]
\item prime numbers $\{p_i: i\in\mathbb{N}_{>0}\}\subset\mathbb{Z}$ or 
\item irreducible polynomials with $1$ as constant coefficient $\{f_j: j\in\mathbb{N}_{>0}\}\subset\mathbb{Q}[x]$.
\end{enumerate}
This gives rise to four different types of prime ideals (see Fig.~\ref{fig:SSpecTop} for an illustration): 
\begin{enumerate}[(a)]
\item height 2 prime ideals of the form $(p_i)$, 
\item height 1 prime ideals of the form $(f_j)$, 
\item the height 1 prime ideal $I_x=(x,\frac{x}{2},\frac{x}{3},\frac{x}{4},\ldots)$, 
\item the height 0 prime ideal $(0)$.
\end{enumerate}

Note that the prime ideal $I_x$ is not finitely generated, reflecting that $S$ is non-Noetherian. We also have
$I_x\subsetneq(p_i)$ for every prime number $p_i$. The ideals $(p_i)$ show that \textbf{Krull's principal ideal theorem}
need not hold in non-Noetherian rings. None of the infinitely many maximal ideals are isolated in $\spec(S)$, thus
Prop.~\ref{prop:Fischer79_3.3} and the whole construction from Sec.~\ref{sec:counterexampleLCMSplitting1} does not apply
here. $S$ is not a UFD, since, e.g., $x=2\cdot(\frac{1}{2}x)=2^2\cdot(\frac{1}{2^2}x)=2^3\cdot(\frac{1}{2^3}x)=\ldots$
acquires arbitrarily many prime factors $2$. This links to $I_x$ not being finitely generated. Nevertheless, there is
``almost a unique prime factorization''. If $a\in S$ has a nonvanishing constant coefficient, it can be written as a
product of powers of irreducibles $a=\pm\prod_ip_i^{r_i}\cdot\prod_jf_j^{s_j}$. If its constant coefficient is $0$, we
have $a=\frac{m}{n}x\cdot x^{u}\cdot\prod_jf_j^{t_j}$. Also note that there are only two units, $S^{\times}=\{\pm1\}$.

For the remainder of this subsection, fix $(a,b)=\left(\frac{x}{q_a},\frac{x}{q_b}\right)$, where $q_a\neq q_b$ are
prime numbers. For concreteness, take $q_a=2, q_b=3$, but note the structural symmetry of this choice. We have
$\gcd(a,b)\sim\frac{x}{6}$ and $\lcm(a,b)\sim x$. Assume for contradiction that $(c,d)$ witnessed lcm splitting for
$(a,b)$. Then $cd\sim\lcm(a,b)\sim x$; wlog $cd=x$. For degree reasons and since the elements of degree $0$ are
integers, any factorization has the form $x=n\cdot \frac{x}{n}$ for $n\in\mathbb{Z}\setminus\{0\}$. By the structural
symmetry, we may assume $c=n,d=\frac{x}{n}$. Since $d=\frac{x}{n}=\frac{x}{n^2}\cdot n=\frac{x}{n^2}\cdot c$,
i.e. $c\mid d$, we get $\gcd(c,d)\stackrel{\text{\tiny (2)}}{\sim}c=n$. The condition $\gcd(c,d)\sim 1$ forces $c=n\in
S^{\times}=\{\pm 1\}$, wlog $c=1$. But then $d=\frac{x}{n}=x\nmid\frac{x}{3}=b$, violating one of the lcm splitting
conditions. Therefore $(a,b)$ does not admit any lcm splitting.

$(a,b)=\left(\frac{x}{2},\frac{x}{3}\right)$ satisfies gf. Set $(c,d)=\left(1,\frac{x}{2}\right)$. Clearly $a=cd$ and
$\gcd(c,b)=\gcd(1,b)\stackrel{\text{\tiny (6), (9)}}{\sim} 1$. Setting $B=\frac{2x}{9}$, we have $b^2=\frac{x^2}{9}=Bd$,
i.e. $b\in\sqrt{(d)}$. By Lem.~\ref{lem:gfImpliesAdequateForPairs}, this implies adequate. By the structural symmetry,
also $(b,a)$ is gf and adequate, too.

\begin{theorem}\label{thm:counterexampleToLCMSplitting2}
For $a=\frac{x}{2},b=\frac{x}{3}\in S\setminus\{0\}$, both $(a,b)$ and $(b,a)$ satisfy the gf and adequate
conditions but fail lcm splitting. In particular, $S$ is a non-adequate, elementary divisor domain that fails lcm splitting.
\end{theorem}

\begin{proof}
All claims except $S$ being non-adequate were proven or referenced in the discussion above. $S$ being non-adequate
follows from the failure of lcm splitting and Thm.~\ref{thm:adequateImpliesLCMSplitting}.
\end{proof}

\bigskip
\section{Final remarks and open questions}
\label{sec:final}

I give further observations and state some open questions.

Let $R$ be a B\'{e}zout domain that fails lcm splitting. The ring $R[x]$ is a GCD domain that is not B\'{e}zout (for a nontrivial
principal ideal $(r)\subset R$, the ideal $(r,x)\subset R[x]$ is not principal), and the same pair
$(a,b)$ that witnessed failure of lcm splitting before again shows failure of lcm splitting in $R[x]$. On the other
hand, let $k$ be any field. Then the polynomial ring $P=k[(x_i)_{i=1}^{n}]$ for $n\geq 2$ is a UFD, hence a GCD domain,
non-B\'{e}zout (e.g. $(x_1,x_2)$ is not principal), but, as a UFD, $P$ satisfies lcm splitting. Thus non-B\'{e}zout GCD
domains may or may not be lcm splitting.

We have seen that adequate domains are always lcm splitting. I have given two non-adequate, elementary divisor domains
that fail lcm splitting. I leave it open whether there are non-adequate, elementary divisor domains that satisfy lcm
splitting. Recently, an example of a B\'{e}zout domain that is not an elementary divisor domain was announced. Is the
example of H\"{a}gg and M\"{o}rtberg~\cite{HM26} lcm splitting? Does failure of the elementary divisor property more
generally imply failure of lcm splitting? Does lcm splitting imply adequate? Note that lcm splitting is symmetric in
$(a,b)$, while adequate (and gf) are not. It would be surprising to me if lcm splitting implied adequate, although I do
not have a good idea what an example of an lcm splitting, non-adequate B\'{e}zout domain should look like.

One could investigate lcm splitting in GCD rings with zero divisors. Gillman and Henriksen~\cite{GH56A} already gave
examples of K-Hermite GCD rings that fail the elementary divisor property.

There are other properties that might be relevant to lcm splitting. How does lcm splitting relate to them?
Anderson et al.~\cite{ADZ04} consider \textbf{almost splitting sets} and refer to work on \textbf{t-splitting
  sets}, which, superficially, seem to be connected to lcm splitting. Anderson~\cite[Sec. 5]{And00}
mentions that more general classes of rings can be \textbf{inertly embedded} into B\'{e}zout domains -- cf. also
Cohn~\cite[Thm. 3.3]{Coh68}. Since inert embeddings leave divisibility properties and gcds mostly unchanged, it is
natural to ask how lcm splitting behaves with respect to such embeddings.

There is a correspondence between B\'{e}zout domains and lattice ordered (abelian) groups. Is there an 
analog of lcm splitting for lattice ordered groups? Some research brought up the \textbf{(strongly) projectable lattice
  ordered groups}, which might be worth investigating further.

Unrelated to lcm splitting, but still interesting: Are there adequate domains that are not gf? Van den Dries and 
Macintyre~\cite[Sec. 2.10]{vDM90} used essentially only topology for their counterexample. Could one reformulate gf and/or
adequate in purely topological terms of spectral spaces/trees? If so, this might open a route to resolve the relation
between gf and adequate.

I exhibited a pair of elements $(\tilde{a},\tilde{b})$ in the ring $R_V$ that violates gf and adequate but is lcm
splitting, and two pairs $(a,b)$, one in $R_V$ and the other in $S$, that violate lcm splitting although both $(a,b)$
and $(b,a)$ satisfy gf and adequate. In general, what are the possible combinations of properties that (pairs of)
elements can have in B\'{e}zout domains? It is not hard to see that as soon as $1\in\{a,b\}$, gf, adequate and lcm
splitting can be satisfied by suitable and obvious choices of $(c,d)$. Therefore, it is impossible for all pairs to
behave badly. See also Rem.~\ref{rem:firstObsLCMSplitting} for the role of $0$ in lcm splitting.

Is it possible to have a pair of elements that violates all three properties, gf, adequate and lcm splitting? One idea is
to take the disjoint union $V''$ of two copies of $V$. The conditions of Prop.~\ref{prop:Fischer79_3.3} are also
satisfied by $V''$, so there is a B\'{e}zout domain $R_{V''}$ with $V''$ as its maximal spectrum. An associated spectral
tree $M''$ should consist of two disjoint copies of $M$ and a single generic point $(0)$. One should be able to set up a
pair $(\hat{a},\hat{b})$ that behaves like $(\tilde{a},\tilde{b})$ on one copy of $M$ in $M''$ and like $(a,b)$ on the
other copy. This should provide a B\'{e}zout domain $R_{V''}$ with a pair $(\hat{a},\hat{b})$ that violates all three
properties, although I have not checked all relevant details. In order to get a pair $(\hat{a},\hat{b})$ such that it and
$(\hat{b},\hat{a})$ violate all three properties, one could try an $R_{V'''}$ associated with three disjoint copies of
$M$.

Finding a B\'{e}zout domain with a pair that violates gf but satisfies adequate could also turn out to be helpful for
the question whether there are domains that are adequate but not gf.

\textbf{Closing thoughts.} It is quite remarkable that lcm splitting seems to be new in commutative algebra, as it
appears to be such an obvious and na\"{i}ve property. At least, I was not able to find substantial literature in
commutative algebra discussing it. I would greatly appreciate any pointers from the community to work treating
this property, possibly under a different name.

\bigskip
\section{Acknowledgment and use of AI}
\label{sec:ackChatbots}

I thank Ash Seifert for the question that inspired lcm splitting, and Andrew Kresch and Elisa Gorla for help with publishing.

Some of the results were developed in dialog with LLMs: the free versions, current at the time of writing, of Anthropic's Sonnet 5
(\url{https://claude.ai/share/21cb9a53-c0db-4ca9-af3e-746fb22aad83}) and OpenAI's ChatGPT
(\url{https://chatgpt.com/share/6ac74f32-08e0-83eb-ad30-f7940cf31be5}). I take full responsibility for the content of
this paper.



\bigskip

\bigskip
\newpage
\appendix
\section{Proofs for basic identities and properties in GCD domains}
\label{app:proofsGCDIdentities}

Here I give proofs of the claims in Prop.~\ref{prop:gcdIdentities}. For convenience, I restate the proposition.

\begin{prop}\label{prop:gcdIdentities}
Let $R$ be a GCD domain and $a,b,c\in R$. Then:
\begin{enumerate}
\item \textbf{definitional divisibility}: $c\mid a \wedge c\mid b\Rightarrow c\mid\gcd(a,b), \quad$ and 
$a\mid c \wedge b\mid c\Rightarrow \lcm(a,b)\mid c$,
\item \textbf{mutual divisibility}: $a\mid b \wedge b\mid a\Rightarrow a\sim b,\quad a\mid b\Rightarrow (\gcd(a,b)\sim a\wedge\lcm(a,b)\sim b)$,
\item \textbf{compatibility with $\sim$}: $\left(a\sim c \wedge b\sim d\right)\Rightarrow \left(\gcd(a,b)\sim \gcd(c,d)
  \wedge \lcm(a,b)\sim \lcm(c,d)\right)$,
\item \textbf{zero behavior}: $\left(a\sim 0 \Leftrightarrow a=0\right), \quad \gcd(a,0)\sim a, \quad$ 
$\left(0\sim \gcd(a,b)\Leftrightarrow (a=0\wedge b=0)\right), \quad$\\ $\lcm(a,0)=0, \quad$ and 
$\left(0\sim \lcm(a,b)\Leftrightarrow (a=0\vee b=0)\right)$, 
\item \textbf{product formula}: $ab\sim \gcd(a,b)\lcm(a,b)$,
\item \textbf{commutativity}: $\gcd(a,b)\sim \gcd(b,a), \quad$ and $\lcm(a,b)\sim \lcm(b,a)$,
\item \textbf{associativity}: $\gcd(a,\gcd(b,c))\sim \gcd(\gcd(a,b),c), \quad$ and $\lcm(a,\lcm(b,c))\sim \lcm(\lcm(a,b),c)$,
\item \textbf{idempotence}: $\gcd(a,a)\sim a, \quad$ and $\lcm(a,a)\sim a$,
\item \textbf{unity}: $c\mid 1\Rightarrow c\sim 1,\quad$ $\gcd(a,b)\mid 1\Rightarrow \gcd(a,b)\sim 1,\quad$ 
$1\not\sim a\mid b\Rightarrow 1\not\sim b,\quad$  $\gcd(a,1)\sim 1,\ \ $ and $\lcm(a,1)\sim a$,
\item \textbf{common factor}: $\gcd(ca,cb)\sim c\gcd(a,b), \quad$ and $\lcm(ca,cb)\sim c\lcm(a,b)$,
\item \textbf{additional factor}: $\gcd(a,b)\mid\gcd(a,bc), \quad$ and $\lcm(a,b)\mid\lcm(a,bc)$,
\item \textbf{coprime elements}: $\gcd(a,b)\sim 1 \wedge a\mid bc\Rightarrow a\mid c, \quad$ 
$\gcd(a,b)\sim 1 \Rightarrow \gcd(a,bc)\sim \gcd(a,c), \quad$ $\gcd(a,b)\sim 1 \Rightarrow \lcm(a,bc)\sim b\lcm(a,c), \quad$ and 
$(a\mid c\wedge b\mid c\wedge \gcd(a,b)\sim 1)\Rightarrow ab\mid c$,
\item \textbf{division by common factor (assume $0\neq c\mid a$ and $c\mid b$, which implies that all following fractions
in $\Frac(R)$ are actually in $R$)}: $\gcd\left(\frac{a}{c},\frac{b}{c}\right) \sim  \frac{\gcd(a,b)}{c}, \quad$ and 
$\lcm\left(\frac{a}{c},\frac{b}{c}\right) \sim  \frac{\lcm(a,b)}{c}$ -- of particular interest will be the special case $c=\gcd(a,b)$,
\item \textbf{powers ($n\in\mathbb{N}, c^{0_{\mathbb{N}}}=1_R$)}: $\gcd(a^n,b^n)\sim \gcd(a,b)^n, \quad$ and $\lcm(a^n,b^n)\sim \lcm(a,b)^n$,
\item \textbf{absorption}: $\gcd(a,\lcm(a,b))\sim a, \quad$ and $\lcm(a,\gcd(a,b))\sim a$,
\item \textbf{distributivity}: $\gcd(a,\lcm(b,c))\sim \lcm(\gcd(a,b),\gcd(a,c)), \quad$ and\\ $\lcm(a,\gcd(b,c))\sim \gcd(\lcm(a,b),\lcm(a,c))$.
\end{enumerate}
\end{prop}

\begin{proof}
\begin{enumerate}
\item This is really in the definition of gcd resp. lcm, so trivial. 
\item From $a\mid b \wedge b\mid a$ we have $a',b' \in R$ such that $a'a=b, b'b=a$.\\
Assume first that (at least) one of $a, b$ is zero, wlog $a=0$. Then, by the general behavior of zero in rings,
$b=a'a=a'\cdot 0=0$ and $b=0=a$, thus $a\sim b$.\\
Otherwise $0\neq a=b'b=b'a'a$ resp. $1\cdot a=b'a'a$. By cancellation in domains, we have $1=b'a'$. Thus $b',a'\in
R^{\times}$ are units, and $a\sim b$.\\
Since $a\mid a$ is trivial, the assumption $a\mid b$ yields, by part (1), $a\mid g=\gcd(a,b)$. But by the other part of the
definition of a gcd, $g$ has to divide each of the two arguments of the gcd as well, i.e. $g\mid a$. By the first part of
mutual divisibility, $\gcd(a,b)=g\sim a$. The proof for lcm follows the same lines.
\item Set $g_1=\gcd(a,b), g_2=\gcd(c,d)$. By definition of $\sim$ there are $a',b'\in R^{\times}$ with $a'a=c, b'b=d$. Thus
$g_1\mid a\mid c$ and $g_1\mid b\mid d$. By part (1), $g_1\mid\gcd(c,d)=g_2$. By symmetry of $\sim$, we also have $g_2\mid
g_1$ and thus by part (2) $g_1\sim g_2$. The proof for lcm follows mutatis mutandis.
\item We use that in any ring $d\cdot 0=0$ for any $d\in R$.\\
$a\sim 0$ implies the existence of $d\in R^{\times}$ with $da=0$. In particular, as a unit in an integral domain,
  $d\neq0$. Since integral domains have no nontrivial zero divisors, we deduce $a=0$ and finish
  ``$\Rightarrow$''. ``$\Leftarrow$'' follows from the assumption and since ``$\sim$'' is reflexive as an equivalence
  relation.\\
It is straightforward to see $\gcd(0,0)=0$, so we may assume $a\neq 0$.\\
We have $g=\gcd(a,0)\mid a$ by definition, but $a=1\cdot a$ and $0=0\cdot a$ show that $a$ is also a common divisor of the
 pair $(a,0)$; hence $a\mid g$ by definition. By mutual divisibility from (2), we get $\gcd(a,0)=g\sim a$.\\  
For the third part, ``$\Leftarrow$'' follows from $\gcd(0,0)=0$. For ``$\Rightarrow$'', we argue by
contraposition. Assume that wlog $a\neq 0$. Since by definition $g=\gcd(a,b)$ is in particular a divisor of
$a$, there is some $g_1\in R$ such that $g_1g=a\neq 0$. By the first part, $g\sim 0$ implies $g=0$. We had
$0=g_1\cdot 0=g_1g=a\neq 0$, a contradiction. Thus $g=\gcd(a,b)\not\sim 0$ demonstrating the third property.\\
Again, in any ring $d\cdot 0=0$ for any $d\in R$, thus the only multiple of $0$ is $0$. Hence $0$ is the only
candidate for an lcm with $0$, and since $0\cdot a=0$ also proves that $0$ is a multiple of $a$, we establish $0$ as the
only common multiple of $a$ and $0$, hence the claimed $\lcm(a,0)=0$.\\
From the previously proven identity and by commutativity (proven below without recourse to this part), we immediately get
``$\Leftarrow$''. For ``$\Rightarrow$'', assume that $a\neq0\neq b$. Then $0\neq ab$ (integral domain) is a common
multiple of $a$ and $b$, thus $l=\lcm(a,b)\mid ab\neq 0$, i.e. there is an $l_1$ with $l_1l=ab\neq 0$. As was already
argued several times, this can only happen if $l\neq 0$; hence, by the first part, $l\not\sim 0$. By contraposition,
$l\sim0$ needs to imply the negation of $a\neq0\neq b$, which yields the fifth and last claim.
\item A proof of a generalized product formula valid for domains that are not necessarily GCD, with the $\gcd$ replaced
  by the \textbf{v-gcd}, may be found in \cite[Thm. 2.1 (2)]{And00}. Since for GCD domains both versions agree, i.e. v-gcd
  is the $\gcd$, this proves the product formula in the above form as a corollary.
\item By the definition of the gcd (logical symmetry in $a$ and $b$), we get $\gcd(a,b)\mid\gcd(b,a)$ and
  $\gcd(b,a)\mid\gcd(a,b)$. Part (2) finishes the proof for gcds. Lcms are dealt with similarly.
\item Set $g_1=\gcd(a,\gcd(b,c)), g_2=\gcd(\gcd(a,b),c)$, and assume that $a,b,c$ are nonzero. By part (4), all gcds
  involved and in particular $g_1, g_2$ are nonzero as well. We first show that
  $g_1\mid g_2$. By definition we have $g_1\mid a, g_1\mid\gcd(b,c)$, thus $g_1$ divides all three, ``$g_1\mid a,b,c$''
  for short. Therefore $g_1\mid\gcd(a,b)$ by definition together with $g_1\mid c$, which in turn by definition yields
  $g_1\mid\gcd(\gcd(a,b),c)=g_2$. By symmetry, $g_2\mid g_1$ as well, and mutual divisibility in part (2) yields $g_1\sim g_2$.\\
If one of $a,b,c$ is zero, the claim follows by a straightforward computation using part (4). E.g., for $a=0$, we
have $\gcd(a,\gcd(b,c))=\gcd(0,\gcd(b,c))\stackrel{\text{\tiny (6), (4)}}{\sim}\gcd(b,c)$ and
$\gcd(\gcd(a,b),c)=\gcd(\gcd(0,b),c)\stackrel{\text{\tiny (6), (4), (3)}}{\sim}\gcd(b,c)$. The claim follows by
transitivity of $\sim$. The cases $b=0$ and $c=0$ follow in the same way.\\
The property for lcm is proved mutatis mutandis.
\item By the two parts of the definition, $\gcd(a,a)\mid a$ and $a\mid \gcd(a,a)$. Mutual divisibility in part (2) implies
  association. The same holds for lcms.
\item From $c\mid 1$, we get some $c'\in R$ with $c'c=1$. This means that $c', c$ are units and associates of $1$,
  in particular $c\sim 1$. Replacing $c$ by $\gcd(a,b)$ yields the second statement.\\
By assumption, there is an $a'\in R$ such that $a'a=b$. Assume, to the contrary, that $1\sim b$. Then there is some $b'\in
R^{\times}$ with $b'b=1$, i.e. $a\mid a'a=b\mid b'b = 1$. By the first part, this yields the contradiction $a\sim 1$ to
the given assumption. Hence $b\not\sim1$, the third claim.\\
By definition, $g=\gcd(a,1)$ divides its constituents, so in particular $g\mid 1$. By the previously proven implication,
this yields $g\sim1$ and thus the claim.\\
Clearly $a$ is a multiple of $a$ (via $1\cdot a$) and of $1$ (via $a\cdot 1$), i.e., a common multiple, and thus
$l=\lcm(a,1)\mid a$. On the other hand, $a$ needs to divide $\lcm(a,1)=l$ by definition of lcm. By part (2), the combination
$l\mid a$ and $a\mid l$ yields $l\sim a$, the last claim of (9).
\item Again, we prove the gcd identity. The lcm identity follows by a dual argument. Again, the case of some $a,b,c$
  being zero can be settled by direct calculations in the spirit of previous parts, so we assume all of them to be nonzero.\\
We first prove a special case of part (13). Set $g=\gcd(a,b)$; we show that $\gcd\left(\frac{a}{g},\frac{b}{g}\right)\sim
1$. First, since $g$ divides both $a$ and $b$, there are $a', b'$ with $a'g=a, b'g=b$, and we have $\frac{a}{g}=a',
\frac{b}{g}=b'\in R\subseteq\Frac(R)$. Now set $g'=\gcd(a',b')$. Then $g'$ divides both $a', b'$, and since divisibility
is multiplicative, $g'g$ divides both $a'g=a, b'g=b$. Part (1) yields $g'g\mid \gcd(a,b)=g$, so there is some $g''$ with
$g''\cdot g'g=g=1\cdot g$, implying $g'$ and $g''$ are units by canceling in domains. This shows $\gcd(a',b')=g'\sim1$
as claimed.\\
With the previous notation, $a=a'g, b=b'g$. Therefore $ca=ca'g, cb=cb'g$ and thus $cg$ is a common divisor of both $ca$
and $cb$, implying $cg\mid\gcd(ca,cb)=G$. Hence there is some $m$ with $mcg=G$. By definition, $G\mid ca$ and $G\mid
cb$, i.e. for some $d_a,d_b\in R\setminus\{0\}$ we have $d_amcg=d_aG=ca=ca'g, d_bmcg=d_bG=cb=cb'g$, and canceling $cg$
on both sides yields $d_am=a', d_bm=b'$. So $m$ divides both $a',b'$; therefore $m\mid\gcd(a',b')\sim1$. By part
(9), $m\sim 1$ is a unit, i.e. $cg=c\gcd(a,b)$ and $G=\gcd(ca,cb)$ are associates as claimed.
\item Set $g=\gcd(a,b)$, so $g\mid a, g\mid b$, which implies $g\mid a, g\mid bc$, which by part (1) yields
  $g\mid\gcd(a,bc)$ which is the first claim.\\
Let $l=\lcm(a,bc)$, thus $a\mid l, bc\mid l$, which implies $a\mid l, b\mid l$, which by part (1) yields
  $\lcm(a,b)\mid l$ which is the other claim.
\item 
By assumption, we have some $a'\in R$ with $a'a=bc$.
By previous parts, we have $a\stackrel{\text{\tiny (8)}}{\sim}\gcd(a,a)\stackrel{\text{\tiny (11)}}{\mid}\gcd(a,a'a)=
\gcd(a,cb)\stackrel{\text{\tiny (11), (6)}}{\mid}\gcd(ca,cb)\stackrel{\text{\tiny (10)}}{\sim}c\gcd(a,b)\sim c\cdot 1 = c$.\\
We have $g=\gcd(a,c)\stackrel{\text{\tiny (11)}}{\mid} \gcd(a,bc)=G$. 
By definition, $G=\gcd(a,bc)$ satisfies $G\mid a$ and $G\mid bc$. Let $d=\gcd(G,b)$. Then we have $d\mid G\mid a,
d\mid b$, i.e. $d\mid\gcd(a,b)\sim 1$, which by part (9) yields $\gcd(G,b)=d\sim1$. The first
part applied to $(a,b,c)$ replaced by $(G,b,c)$ (conditions $\gcd(G,b)\sim1$ and $G\mid bc$ are satisfied) yields $G\mid
c$. Combine this with $G\mid a$ to get $G\mid\gcd(a,c)=g$. Combining $g\mid G$ and $G\mid g$ by part (2) yields $g\sim G$, i.e. the claim.\\ 
For the third part of (12), we use the previous part of (12) just proven and the product formula (5) (treating zero cases separately,
as before).
\[
\lcm(a,bc)\stackrel{\text{\tiny (5)}}{\sim}
\frac{abc}{\gcd(a,bc)}\stackrel{\text{\tiny (12)}}{\sim}
\frac{abc}{\gcd(a,c)}=
b\cdot \frac{ac}{\gcd(a,c)}\stackrel{\text{\tiny (5)}}{\sim}
b\cdot \lcm(a,c).
\]
There are $a',b'\in R$ with $a'a=c=b'b$, i.e. $a\mid b'b$. The first statement of (12) together with the assumption
$\gcd(a,b)\sim1$ yields $a\mid b'$. Since divisibility is multiplicative, we have $ab\mid b'b=c$, which is the claim.
\item We already proved a special case of part (13) in (10) for $c=\gcd(a,b)$. Note that an arbitrary common divisor $c$
  of $a,b$ also satisfies $c\mid\gcd(a,b)=g$, i.e. there are $a', b', c'\in R$ with $a'g=a, b'g=b, c'c=g$. Then we
  calculate
\begin{gather*}
\gcd\left(\frac{a}{c},\frac{b}{c}\right) = 
\gcd\left(\frac{a'g}{c},\frac{b'g}{c}\right) = 
\gcd\left(a'\frac{g}{c},b'\frac{g}{c}\right) = 
\gcd\left(a'c',b'c'\right) \stackrel{\text{\tiny (10)}}{\sim} \\
c'\cdot\gcd\left(a',b'\right) =
\frac{g}{c}\cdot\gcd\left(a',b'\right) =
\frac{g\cdot\gcd\left(a',b'\right)}{c} \stackrel{\text{\tiny (10)}}{\sim}
\frac{\gcd\left(a'g,b'g\right)}{c} =
\frac{\gcd(a,b)}{c}.
\end{gather*}
Since $c\neq 0$, all the fractions are actually defined in $R\subseteq\Frac(R)$.\\
As before, treating zero cases separately and taking care that all fractions
stay in $R$, we obtain the other assertion from the
product formula (5):
\[
\lcm\left(\frac{a}{c},\frac{b}{c}\right) \stackrel{\text{\tiny (5)}}{\sim} 
\frac{\frac{a}{c}\cdot\frac{b}{c}}{\gcd\left(\frac{a}{c},\frac{b}{c}\right)}\sim
\frac{\frac{a}{c}\cdot\frac{b}{c}}{\frac{\gcd(a,b)}{c}}=
\frac{ab}{c\gcd(a,b)} \stackrel{\text{\tiny (5)}}{\sim} 
\frac{\lcm(a,b)}{c}.
\]
\item Treat the special cases of $a,b\in R$ and $n\in\mathbb{N}$ not all being nonzero separately. The convention
  $c^0=1$ works out in all cases, even for $0^0=1$.\\
Let $g=\gcd(a,b)$. Then $a=a'g, b=b'g$ for some $a',b' \in R$. Hence $a^n=a'^ng^n, b^n=b'^ng^n$ and therefore
$g^n\mid\gcd(a^n,b^n)$, meaning $\frac{\gcd(a^n,b^n)}{g^n}$ is well-defined in $R$.\\
By the special case of part (13) proven in (10), we have $\gcd\left(\frac{a}{g},\frac{b}{g}\right)\sim1$. Note that we
can assume $n\geq 1$. Inductively applying the second identity of part (12), we get 
\[\gcd\left(\frac{a}{g},\frac{b^n}{g^n}\right)=
\gcd\left(\frac{a}{g},\left(\frac{b}{g}\right)^n\right)\stackrel{\text{\tiny (12)}}{\sim}
\gcd\left(\frac{a}{g},\left(\frac{b}{g}\right)^{n-1}\right)\stackrel{\text{\tiny (12)}}{\sim}
\ldots\stackrel{\text{\tiny (12)}}{\sim}
\gcd\left(\frac{a}{g},\left(\frac{b}{g}\right)^{1}\right)\sim1.
\] 
Similarly, using the association just proven and (6), we get
\[\gcd\left(\frac{a^n}{g^n},\frac{b^n}{g^n}\right)=
\gcd\left(\left(\frac{a}{g}\right)^n,\frac{b^n}{g^n}\right)\stackrel{\text{\tiny (12)}}{\sim}
\gcd\left(\left(\frac{a}{g}\right)^{n-1},\frac{b^n}{g^n}\right)\stackrel{\text{\tiny (12)}}{\sim}
\ldots\stackrel{\text{\tiny (12)}}{\sim}
\gcd\left(\left(\frac{a}{g}\right)^{1},\frac{b^n}{g^n}\right)\sim1.
\] 
Then we get
\[
\frac{\gcd(a^n,b^n)}{g^n}\stackrel{\text{\tiny (13)}}{\sim}
\gcd\left(\frac{a^n}{g^n},\frac{b^n}{g^n}\right)\sim1.
\]
Multiplying both sides by $g^n$ and canceling yields the first identity. The lcm identity follows by applying the product 
formula (5), treating zero cases separately, as seen before.
\item Set $l=\lcm(a,b), g=\gcd(a,l)$. Since $a$ is an argument of the outer gcd, we have by definition $g\mid a$. We also
  have $a\mid a$ and, by definition of lcm, $a\mid l$, which, by (1), combines to $a\mid g$. Mutual divisibility (2) yields the
  claim.\\ 
The second absorption statement is proven analogously.
\item First, we treat the zero cases explicitly: 
\begin{gather*}
a=0: \gcd(0,\lcm(b,c))\stackrel{\text{\tiny (4)}}{\sim}
\lcm(b,c)\stackrel{\text{\tiny (3), (4), (6) twice}}{\sim}
\lcm(\gcd(0,b),\gcd(0,c)),\\
b=0: \gcd(a,\lcm(0,c))\stackrel{\text{\tiny (4), (6)}}{=}
\gcd(a,0)\stackrel{\text{\tiny (4)}}{\sim}
a\stackrel{\text{\tiny (15)}}{\sim}
\lcm(a,\gcd(a,c))\stackrel{\text{\tiny (3), (4)}}{\sim}
\lcm(\gcd(a,0),\gcd(a,c)),\\
c=0: \gcd(a,\lcm(b,0))\stackrel{\text{\tiny (4)}}{=}
\gcd(a,0)\stackrel{\text{\tiny (4)}}{\sim}
a\stackrel{\text{\tiny (6), (15)}}{\sim}
\lcm(\gcd(a,b),a)\stackrel{\text{\tiny (3), (4)}}{\sim}
\lcm(\gcd(a,b),\gcd(a,0)).
\end{gather*}
Next, assume $g=\gcd(a,\gcd(b,c))\sim1$ and set $g_b=\gcd(a,b), g_c=\gcd(a,c), LHS=\gcd(a,\lcm(b,c)), RHS=\lcm(g_b,g_c)$.\\ 
We first show $RHS\mid LHS$. We have $g_b\mid a, g_b\mid b\mid\lcm(b,c)$, thus $g_b\mid LHS$. From $g_c\mid a, g_c\mid
c\mid\lcm(b,c)$ follows $g_c\mid LHS$. I.e., $LHS$ is a common multiple of $g_b, g_c$, and by part (1), we have $RHS=\lcm(g_b,g_c)\mid LHS$.\\
We prove an intermediate result, which we will use shortly. For $x,y,z\in R\setminus\{0\}$, we have $(x\mid
yz\wedge\gcd(\gcd(x,y),\gcd(x,z))\sim1) \Rightarrow x\sim\gcd(x,y)\gcd(x,z)$.\\
Set $g_y=\gcd(x,y), g_z=\gcd(x,z)$. By the definition of gcds, we have $g_y\mid x$ and $g_z\mid x$, and by assumption
$\gcd(g_y,g_z)\sim1$, so the last part of (12) implies $g_yg_z\mid x$, i.e. there is an $x'\in R\setminus\{0\}$ with
$x'g_yg_z=x$. By definition of gcds, we have $g_y\mid y, g_z\mid z$, hence there are $y',z'\in R\setminus\{0\}$ with
$y'g_y=y, z'g_z=z$. We have $g_y=\gcd(x,y)=\gcd(x'g_yg_z,y'g_y)\stackrel{\text{\tiny (10)}}{\sim}g_y\gcd(x'g_z,y')$, so
$\gcd(x'g_z,y')\sim1$ and symmetrically $\gcd(x'g_y,z')\sim1$. Now, by assumption, $x'g_yg_z=x\mid
yz=y'g_yz'g_z$, i.e., after canceling $g_yg_z$, we have $x'\mid y'z'$. Since $\gcd(x',y')\stackrel{\text{\tiny
    (6), (11)}}{\mid}\gcd(x'g_z,y')\sim1$, part (9) implies $\gcd(x',y')\sim1$, which, together with $x'\mid y'z'$, implies via part (12)
that $x'\mid z'$. Symmetrically, we get $\gcd(x',z')\sim1$, which, together with $x'\mid z'=z'\cdot 1$, yields again via (12)
that $x'\mid 1$, and by (9) we have $x'\sim1$. Therefore, the previous relation $x=x'g_yg_z\sim1\cdot
g_yg_z=g_yg_z=\gcd(x,y)\gcd(x,z)$ is exactly the intermediate result.\\
Using commutativity (6) and associativity (7) repeatedly and finally the intermediate assumption of coprimality
$g=\gcd(a,\gcd(b,c))\sim1$, we have: 
\[\gcd(g_b,g_c)=
\gcd(\gcd(a,b),\gcd(a,c))\stackrel{\text{\tiny (6), (7)}}{\sim}
\gcd(\gcd(a,a),\gcd(b,c))\stackrel{\text{\tiny (8)}}{\sim}
\gcd(a,\gcd(b,c))\sim1.
\]
Thus $RHS\stackrel{\text{\tiny (5)}}{\sim}\frac{g_bg_c}{\gcd(g_b,g_c)}\sim\frac{g_bg_c}{1}=g_bg_c$. 
Next $LHS=\gcd(a,\lcm(b,c))\mid\lcm(b,c)\mid bc$, which we call $\pentagram_1$. By definition of gcd, $LHS\mid a$, i.e. there is $a' \in R$ with
$a'LHS=a$, hence by (6) and (11): $\gcd(LHS,b)\mid\gcd(a'LHS,b)=\gcd(a,b)=g_b$ and symmetrically $\gcd(LHS,c)\mid
g_c$. By (6) and another twofold application of (11), this yields $\pentagram_2$:
$\gcd(\gcd(LHS,b),\gcd(LHS,c))\mid\gcd(g_b,g_c) \sim 1$.\\ 
Setting $x=LHS,y=b,z=c$ in the intermediate result, $\pentagram_1$ and $\pentagram_2$ are the two conditions needed to apply
the intermediate result, which gives $LHS\sim\gcd(LHS,b)\gcd(LHS,c)\mid g_bg_c=RHS$.\\
Now mutual divisibility (2) proves the claim $LHS\sim RHS$ in the coprime case.\\
Let $g=\gcd(a,\gcd(b,c))$ (possibly $\not\sim 1$). Since the zero cases are already resolved, by (4) we may assume
$g\neq 0$. There are $a', b', c'\in R\setminus\{0\}$ such that $a'g=a, b'g=b, c'g=c$. Set
$g'=\gcd(a',\gcd(b',c'))\stackrel{\text{\tiny (4)}}{\neq} 0$. Tracing the definitions of gcd (details left to the
reader), we deduce $g'g\mid c'g=c$ and $g'g\mid b'g=b$, hence $g'g\mid\gcd(b,c)$, and further $g'g\mid a'g=a$, thus
$g'g\mid\gcd(a,\gcd(b,c))=g$, i.e. there is a $g''$ with $g''g'g=g$. Canceling $g\neq 0$ in the domain $R$ yields
$g''g'=1$, hence $g'$ is a unit and $g'\sim 1$. Thus $a', b', c'$ satisfy the coprime condition $g'\sim1$, and we may
apply the special case that we just proved. We have 
\begin{gather*}
\gcd(a,\lcm(b,c))=
\gcd(a'g,\lcm(b'g,c'g))\stackrel{\text{\tiny (10) twice}}{\sim}
g\gcd(a',\lcm(b',c'))\stackrel{\text{\tiny coprime case}}{\sim}\\
g\lcm(\gcd(a',b'),\gcd(a',c'))\stackrel{\text{\tiny (10) thrice}}{\sim}
\lcm(\gcd(a'g,b'g),\gcd(a'g,c'g))=
\lcm(\gcd(a,b),\gcd(a,c)).
\end{gather*}
For the second distributivity law, we again treat the three zero cases explicitly. We may assume that all terms
involved are now nonzero. We apply the distributivity result just proven with $(a,b,c)$ replaced by $(\lcm(a,b),a,c)$ and
then absorption (15):
\[
\gcd(\lcm(a,b),\lcm(a,c))\stackrel{\text{\tiny dist}}{\sim}
\lcm(\gcd(\lcm(a,b),a),\gcd(\lcm(a,b),c))\stackrel{\text{\tiny (3), (6), (15)}}{\sim}
\lcm(a,\gcd(\lcm(a,b),c)).
\]
Apply the previous distributivity again with $(a,b,c)$ replaced by $(c,a,b)$:
\[
\gcd(\lcm(a,b),c)\stackrel{\text{\tiny (6)}}{\sim}
\gcd(c,\lcm(a,b))\stackrel{\text{\tiny dist}}{\sim}
\lcm(\gcd(c,a),\gcd(c,b))\stackrel{\text{\tiny (6) twice}}{\sim}
\lcm(\gcd(a,c),\gcd(b,c)).
\]
Substituting this into the previous equation yields
\begin{gather*}
\gcd(\lcm(a,b),\lcm(a,c))\sim
\lcm(a,\gcd(\lcm(a,b),c))\stackrel{\text{\tiny subst}}{=}
\lcm(a,\lcm(\gcd(a,c),\gcd(b,c)))\stackrel{\text{\tiny (7)}}{\sim}\\
\lcm(\lcm(a,\gcd(a,c)),\gcd(b,c))\stackrel{\text{\tiny (15)}}{\sim}
\lcm(a,\gcd(b,c)).
\end{gather*}
This is the second distributivity law.
\end{enumerate}
\end{proof}

\bigskip
\newpage
\section{Uniqueness of $(c,d)$ up to associates in the gf and adequate properties}
\label{sec:uniquenessInGFAdequateSplitting}

First, I repeat an argument from \cite[p. 460]{LP20} for easier reference.

\begin{lemma}\label{lem:gfImpliesAdequateForPairs}
Let $R$ be a B\'{e}zout domain and $a,b\in R\setminus\{0\}$. If $c,d\in R$ satisfy the gf condition for $(a,b)$, then
they satisfy the adequate condition.
\end{lemma}

\begin{proof}
Recall the gf condition: $a=cd$, $\gcd(c,b)\sim 1$, and $b\in\sqrt{(d)}$. The first two properties of gf and adequate
coincide. Let $d'\in R\setminus R^{\times}$ with $d'\mid d$, i.e. there is some $d''\in R$ with
$d'd''=d$. From $b\in\sqrt{(d)}$, we get $n\in\mathbb{N}_{>0}, B\in R$ with $b^n=dB$, and thus $b^n=dB=d'd''B$. Assume
for contradiction that $\gcd(d',b)\sim 1$. Then (note that $d'^{n-1}\in R$ is well-defined since $n>0$)
\[
d'\mid
d'\gcd(d'^{n-1},d''B)\stackrel{\text{\tiny (10)}}{\sim}
\gcd(d'^n,d'd''B)=
\gcd(d'^n,b^n)\stackrel{\text{\tiny (14)}}{\sim}
\gcd(d',b)^n\sim
1^n=
1.
\]
By (9) we had $d'\sim 1$ contradicting $d'\not\in R^{\times}$. Thus $\gcd(d',b)\not\sim 1$. This finishes the proof.
\end{proof}

Next I prove that $(c,d)$ satisfying the adequate (resp. gf) condition for a pair $(a,b)$ are unique up to associates.

\begin{lemma}\label{lem:uniquenessInAdequateSplitting}
Let $R$ be a B\'{e}zout domain. Given $a, b\in R\setminus\{0\}$ and two pairs $(c_1, d_1), (c_2, d_2) \in R^2$
satisfying the adequate conditions for $(a,b)$, we have $c_1\sim c_2$ and $d_1\sim d_2$.
\end{lemma}

\begin{proof}
Note that $0\neq a=c_1d_1=c_2d_2$ implies that all occurring denominators are nonzero. 

Set $d=\gcd(d_1,d_2), \tilde{d}_1=\frac{d_1}{d}, \tilde{d}_2=\frac{d_2}{d}$, and $c=\frac{a}{d}$. We start with some
auxiliary calculations.
\begin{gather*}
\gcd(\tilde{d}_1,\tilde{d}_2)=
\gcd\left(\frac{d_1}{d},\frac{d_2}{d}\right)\stackrel{\text{\tiny (13)}}{\sim}
\frac{\gcd\left(d_1,d_2\right)}{d}=
\frac{d}{d}=
1,\\
\tilde{d}_1\tilde{d}_2\lcm(c_1,c_2)\stackrel{\text{\tiny (10)}}{\sim}
\lcm(c_1\tilde{d}_1\tilde{d}_2,c_2\tilde{d}_1\tilde{d}_2)=
\lcm\left(\frac{a}{d_1}\cdot\tilde{d}_1\tilde{d}_2,\frac{a}{d_2}\cdot\tilde{d}_1\tilde{d}_2\right)=
\lcm\left(\frac{a\tilde{d}_1\tilde{d}_2}{d\tilde{d}_1},\frac{a\tilde{d}_1\tilde{d}_2}{d\tilde{d}_2}\right)\stackrel{\text{\tiny (10)}}{\sim}\\
\frac{a}{d}\lcm\left(\tilde{d}_2,\tilde{d}_1\right)\stackrel{\text{\tiny (6), (5)}}{\sim}
\frac{a}{d}\cdot\frac{\tilde{d}_1\tilde{d}_2}{\gcd\left(\tilde{d}_1,\tilde{d}_2\right)}\sim
c\cdot\frac{\tilde{d}_1\tilde{d}_2}{1}=
c\tilde{d}_1\tilde{d}_2.
\end{gather*}
Canceling in the domain $R$ yields $c\sim\lcm(c_1,c_2)$. Now we check that $(c,d)$ also satisfies the adequate conditions
for $(a,b)$. The definition of $c$ immediately gives $a=cd$.
\[
\gcd(c,b)\stackrel{\text{\tiny (3)}}{\sim}
\gcd(\lcm(c_1,c_2),b)\stackrel{\text{\tiny (16), (6)}}{\sim}
\lcm(\gcd(c_1,b),\gcd(c_2,b))\stackrel{\text{\tiny adequate twice}}{\sim}
\lcm(1,1)\stackrel{\text{\tiny (8)}}{\sim}
1.
\]
By the definition of gcd, we have $d=\gcd(d_1,d_2)\mid d_1$. Given any $d'\in R\setminus R^{\times}$ with $d'\mid d$, we
immediately get $d'\mid d\mid d_1$ and can apply the adequate conditions for $(c_1,d_1)$. We get $\gcd(d',b)\not\sim1$
as required. Thus $(c,d)$ is adequate for $(a,b)$.

Next, consider $d_1'=\frac{c}{c_1}$. 
Since $d_1'=\frac{c}{c_1}\mid c$, we have 
\[
\gcd(d_1',b)\stackrel{\text{\tiny (6), (11)}}{\mid}
\gcd(c,b)\stackrel{\text{\tiny $(c,d)$ adequate}}{\sim} 
1\stackrel{\text{\tiny (9)}}{\Rightarrow} 
\gcd(d_1',b)\sim
1.
\]
From $c\mid a$, we get $d_1'=\frac{c}{c_1}\mid\frac{a}{c_1}=d_1$. Assume $d_1'\not\in R^{\times}$. Then $(c_1,d_1)$
adequate for $(a,b)$ implies $\gcd(d_1',b)\not\sim1$, a contradiction. Thus $d_1'\sim 1$. From $d_1'c_1=c$ we then get
$c_1\sim c$. By symmetry, also $c\sim c_2$ holds, and by transitivity of $\sim$ finally $c_1\sim c_2$. 
$d_1=\frac{a}{c_1}\sim\frac{a}{c_2}=d_2$ finishes the proof.
\end{proof}

\begin{cor}\label{cor:uniquenessInGFSplitting}
Let $R$ be a B\'{e}zout domain. Given $a, b\in R\setminus\{0\}$ and two pairs $(c_1, d_1), (c_2, d_2) \in R^2$
satisfying the gf conditions for $(a,b)$, we have $c_1\sim c_2$ and $d_1\sim d_2$.
\end{cor}

\begin{proof}
We apply Lem.~\ref{lem:gfImpliesAdequateForPairs} to show that the $(c_i,d_i)$ have adequate properties for $(a,b)$ and
then Lem.~\ref{lem:uniquenessInAdequateSplitting} to get the association.
\end{proof}

\bigskip
\section{Good factorization implies lcm splitting}
\label{sec:gfLcmSplitting}

I originally found a proof that gf implies lcm splitting. The global structure is similar to the proof of adequate
implies lcm splitting in Thm.~\ref{thm:adequateImpliesLCMSplitting}, but the details about the radical are different and
could be interesting for algorithmic versions.

\begin{lemma}\label{lem:gfFactorsAreCoprime}
Let $R$ be a GCD domain. Given $a, b\in R\setminus\{0\}$ and some $c, d \in R$ satisfying the gf conditions,
i.e. $a = cd, \gcd(c,b)\sim1$ and $b\in\sqrt{(d)}$, we have $\gcd(c,d)\sim1$.
\end{lemma}

\begin{proof}
$b\in\sqrt{(d)}$ implies by definition of the radical the existence of $n\in\mathbb{N}_{>0}, B\in R$ with $b^n=dB$.
\[
\gcd(c,d)\stackrel{\text{\tiny (11)}}{\mid}
\gcd(c,dB)=
\gcd(c,b^n)\stackrel{\text{\tiny (6), (11)}}{\mid}
\gcd(c^n,b^n)\stackrel{\text{\tiny (14)}}{\sim}
\gcd(c,b)^n\sim
1^n=1.
\]
This proves the lemma by (9).
\end{proof}

\begin{theorem} \label{thm:gfImpliesLCMSplitting}
A gf B\'{e}zout domain $R$ is lcm splitting.
\end{theorem}

\begin{proof}
Let $a, b \in R\setminus\{0\}$ be arbitrary (cf. Rem.~\ref{rem:firstObsLCMSplitting}). Set
$g=\gcd(a,b), l=\lcm(a,b)$. From (5), $ab\sim gl$, we deduce $l\sim\frac{ab}{g}=\frac{a}{g}g\frac{b}{g}$.

Now apply the gf property to the pair $\left(g, \frac{b}{g}\right)$ to get $g=g_1g_2$ with
$\gcd\left(g_1,\frac{b}{g}\right)\sim1$ and $\frac{b}{g}\in\sqrt{(g_2)}$. Apply gf a second time to $\left(g_2,
\frac{a}{g}\right)$ to get $g_2=g_1'g_2'$ with $\gcd\left(g_2',\frac{a}{g}\right)\sim1$ and
$\frac{a}{g}\in\sqrt{(g_1')}$. Overall, we have $l\sim\frac{a}{g}g\frac{b}{g}=\frac{a}{g}g_1g_1'g_2'\frac{b}{g}$, and we
will show some coprimality properties among these factors:
\begin{enumerate}
\item[(A)] $\gcd\left(\frac{a}{g},g_2'\right)\sim1$ follows directly from the 2nd application of gf and (6). 
\item[(B)] Setting $c=g$ in (13), we get $\gcd\left(\frac{a}{g},\frac{b}{g}\right)\sim\frac{\gcd(a,b)}{g}=\frac{g}{g}=1$. 
\item[(C)] 
$\gcd(g_1,g_2')\stackrel{\text{\tiny (11)}}{\mid}
\gcd(g_1,g_1'g_2')=
\gcd(g_1,g_2)\sim 1$, 
where the last step follows from Lem.~\ref{lem:gfFactorsAreCoprime}. By (9) we get $\gcd(g_1,g_2')\sim1$.
\item[(D)] $\gcd\left(g_1,\frac{b}{g}\right)\sim1$ follows directly from the 1st application of gf.
\item[(E)] By the 1st application of gf, we have some $n\in\mathbb{N}_{>0}, B\in R$ such that $g_2B=\left(\frac{b}{g}\right)^n$. 
By the 2nd application of gf, we have some $m\in\mathbb{N}_{>0}, A\in R$ such that $g_1'A=\left(\frac{a}{g}\right)^m$. Thus 
\begin{gather*}
g_1'\mid g_1'g_2'=g_2\mid g_2B=\left(\frac{b}{g}\right)^n\mid\left(\frac{b}{g}\right)^{mn}\quad \text{ and } \quad
g_1'\mid g_1'A=\left(\frac{a}{g}\right)^m\mid\left(\frac{a}{g}\right)^{mn}.
\end{gather*}
By (1), we have $g_1'\mid
\gcd\left(\left(\frac{a}{g}\right)^{mn},\left(\frac{b}{g}\right)^{mn}\right)\stackrel{\text{\tiny (14)}}{\sim}
\gcd\left(\frac{a}{g},\frac{b}{g}\right)^{mn}\stackrel{\text{\tiny (B)}}{\sim}
1^{mn}=1$. By (9): $g_1'\sim1$.
\end{enumerate}

By (E), only four nontrivial factors remain: $l\sim\frac{a}{g}g_1g_2'\frac{b}{g}$. We set $c=\frac{a}{g}g_1$ and $d=g_2'\frac{b}{g}$. 
Now we apply (12) several times based on the coprimality results (A)--(D):
\[
(\alpha) \gcd\left(\frac{a}{g},g_2'\right)\stackrel{\text{\tiny (A)}}{\sim}1 \stackrel{\text{\tiny (12)}}{\Rightarrow}
\gcd\left(\frac{a}{g},g_2'\frac{b}{g}\right)\sim
\gcd\left(\frac{a}{g},\frac{b}{g}\right)\stackrel{\text{\tiny (B)}}{\sim}
1,
\]
\[
(\beta) \gcd\left(g_1,g_2'\right)\stackrel{\text{\tiny (C)}}{\sim}1 \stackrel{\text{\tiny (12)}}{\Rightarrow}
\gcd\left(g_1,g_2'\frac{b}{g}\right)\sim
\gcd\left(g_1,\frac{b}{g}\right)\stackrel{\text{\tiny (D)}}{\sim}
1,
\]
\[
(\gamma) \gcd\left(\frac{a}{g},g_2'\frac{b}{g}\right)\stackrel{\text{\tiny $(\alpha)$}}{\sim}1 \stackrel{\text{\tiny (6), (12)}}{\Rightarrow}
\gcd\left(c,d\right)=
\gcd\left(\frac{a}{g}g_1,g_2'\frac{b}{g}\right)\sim
\gcd\left(g_1,g_2'\frac{b}{g}\right)\stackrel{\text{\tiny $(\beta)$}}{\sim}
1.
\]
Having shown $\gcd\left(c,d\right)\sim1$, we verify three further properties.
$cd=\frac{a}{g}g_1g_2'\frac{b}{g}\sim l=\lcm(a, b)$ by construction. From $a=\frac{a}{g}g=\frac{a}{g}g_1g_2$, we
get $c=\frac{a}{g}g_1\mid a$, and finally, from $b=g\frac{b}{g}=g_1g_1'g_2'\frac{b}{g}$, we get $d=g_2'\frac{b}{g}\mid b$.
Since $a,b$ were arbitrary, $R$ is an lcm splitting B\'{e}zout domain.
\end{proof}

\bigskip
\section{Examples of non-UFD B\'{e}zout domains that are lcm splitting}
\label{sec:furtherExamples}

I will give three examples of non-UFD B\'{e}zout domains that are lcm splitting.
\begin{enumerate}
\item Let $k$ be a field and $R_1=\bigcup_{n=1}^{\infty}k[[x^{1/n}]]$, i.e. \textbf{formal power series} with rational exponents of bounded 
denominators, also called \textbf{Puiseux series}. It is not hard to see that any $r\in R_1\setminus\{0\}$ can be written in the form
  $r=cx^{q}$ for $c\in R_1^{\times}$ a unit and $q\in\mathbb{Q}_0^+$. In this ring, divisibility is a total order. 
  In fact, $R_1$ is a valuation domain with valuation in $\mathbb{Q}_0^+$ and B\'{e}zout. 
  Explicitly, given any $a=c_ax^{q_a},b=c_bx^{q_b}\in R_1\setminus\{0\}$ and assuming wlog
  $q_a\geq q_b$, we have $\gcd(a,b)\sim b, \lcm(a,b)\sim a$. 
  An lcm splitting is given by $c=a,d=1$, which is easy to verify. It is, however, not a UFD, since it lacks
  atoms. Any candidate for an irreducible element $0\neq r=cx^{q}$ must have $q>0$ (otherwise it is a unit), and
  $r=(cx^{q/2})\cdot x^{q/2}$ contradicts irreducibility. Since there are no irreducibles, there cannot exist a
  factorization into them.
\item Let $R_2=\mathcal{O}(\mathbb{C})$ be the \textbf{ring of entire functions} on the complex plane. $R_2$ is a
  B\'{e}zout domain by a well-known result of Helmer~\cite{Hel40}. By the \textbf{Weierstra{\ss} factorization theorem}
  we can express any entire function as an infinite product of elementary factors $E_k$. Hence $R_2$ is not a UFD, as
  factorizations in a UFD are necessarily finite. Since this factorization is, however, essentially unique, it can also
  be used to construct lcm splittings. I leave the general details to the reader, but give one illustrative example of
  the shortcomings of the effective procedure (algorithm for a UFD) presented on p.~\pageref{page:lcmSplittingAlgo}.

By the Weierstra{\ss} factorization theorem, we can construct holomorphic functions $f,g$ using elementary factors
$E_{k}$ as follows:
\[
f(z)=\prod\limits_{n=1}^\infty E_{p_n}\left(\frac{z}{a_n}\right),\quad
g(z)=\prod\limits_{n=1}^\infty E_{p_n}\left(\frac{z}{b_n}\right),
\]
where $(a_n)=(1,2,3,\ldots)$ is the sequence of positive integers, 
$(b_n)=(1,2,2,3,3,3,4,\ldots,k-1,\stackrel{k-\text{fold}}{\overbrace{k,\ldots,k}},k+1,\ldots)$ and 
$(p_n)=(2,2,2,\ldots)$ is the constant sequence $2$. The idea is that $f$ has zeros exactly at the positive integers 
$z_i=i\in\mathbb{Z}_{>0}$ with constant multiplicity $1$, and $g$ has the same zeros but with increasing multiplicities $i$. Set 
\[
a=a(z)=g\left(\frac{1}{2}(z+1)\right)\cdot g\left(\frac{1}{2}z\right)\cdot f\left(\frac{1}{2}(z+1)\right),
b=b(z)=g\left(\frac{1}{2}(z+1)\right)\cdot g\left(\frac{1}{2}z\right)\cdot f\left(\frac{1}{2}z\right).
\]
The zeros of $a,b$ are still at the positive integers. For an odd positive integer $z_o$, the multiplicity of $a$ is $1$ more than
that of $b$, i.e. $\mult_{z_o}(a)=\mult_{z_o}(b)+1$, and vice versa for even positive integers $z_e$:
$\mult_{z_e}(a)+1=\mult_{z_e}(b)$. Moreover, the multiplicities increase linearly:
$\mult_{2n-1}(b)=\mult_{2n}(a)=n$. The procedure on p.~\pageref{page:lcmSplittingAlgo} produces $c_{n},d_{n}$ in step
$n$ such that for an even positive integer $z_k=2k$, we have $\mult_{z_k}\gcd(c_{n},d_{n})>0$ if and only if $k\geq 2^{n}$. Since
we can take $k$ arbitrarily large, we see that this approach will never terminate globally. However, viewing each ``finite
place'' $z_k$ in isolation, it does terminate in the sense that $\mult_{z_k}\gcd(c_{n},d_{n})=0$ for fixed $k$ and all
sufficiently large $n$.

Using the factorizations directly, we set 
\[
c=g\left(\frac{1}{2}(z+1)\right)\cdot f\left(\frac{1}{2}(z+1)\right),
d=g\left(\frac{1}{2}z\right)\cdot f\left(\frac{1}{2}z\right).
\]
We claim that $(c,d)$ is an lcm splitting of $(a,b)$. Let $g=\gcd(c,d)$. If $g$ were not a unit, then $g$ would have a
zero at some $z_0\in\mathbb{C}$. However, all zeros of $c$ are at the odd positive integers and all zeros of $d$ are at
the even positive integers, so there cannot be a common zero $z_0$. Thus $g$ is a unit and can be chosen to be
$g=1$. The divisibility properties $c\mid a$ and $d\mid b$ are obvious. The lcm condition $cd\sim\lcm(a,b)$ follows
similarly by comparing orders of vanishing at the relevant points.

\item Let $R_3=\bar{\mathbb{Z}}$ be the ring of algebraic integers. According to \cite[Sec. 2.11]{vDM90}, it is a gf
  B\'{e}zout domain. Thus Thm.~\ref{thm:adequateImpliesLCMSplitting} and ``gf $\Rightarrow$ adequate''
  (cf. Rem.~\ref{rmk:propertiesOfGCDDomains}) immediately imply that it has lcm splittings. It is, however, instructive to
  look at gcds and lcm splittings explicitly.

Given explicit $a,b\in R_3\setminus\{0\}$, the $\gcd(a,b)$ might not exist in the ring of integers $\mathfrak{o}_K$ of 
$K=\mathbb{Q}(a,b)$. But there is always an explicitly constructible finite extension field $L/K$ such that $\gcd(a,b)$ 
exists in $\mathfrak{o}_L$. This does not mean that $\mathfrak{o}_L$ is itself a GCD domain or even B\'{e}zout, but 
only that enough prime ideals of $\mathfrak{o}_K$ become products of principal ideals of $\mathfrak{o}_L$ such that 
$\gcd(a,b)$ exists and $a,b$ admit suitable factorizations (not necessarily by irreducibles), facilitating the 
computation of $\gcd(a,b)$. Using these factorizations, it is possible to find good factorization splittings and 
ultimately suitable $c,d$ for lcm splitting.

Since irreducible elements of some $\mathfrak{o}_K$ split further in suitable extensions $\mathfrak{o}_L$, and this does
not stop after finitely many steps, $R_3=\bar{\mathbb{Z}}$ is not a UFD. In the spirit of an argument for $S$ in
Sec.~\ref{sec:counterexampleLCMSplitting2}, one can also consider any integers $a,k\in\mathbb{N}_{>1}$ and the infinite
sequence of refining factors $\left(a^{1/k},a^{1/k^2},a^{1/k^3},\ldots\right)$ to see directly that it is not a UFD.

In this example, the ring $R_3$ has the structure of a union: $\bar{\mathbb{Z}}=\bigcup_i\mathfrak{o}_i$, where one can even
assume that it is an ascending chain, i.e. $\mathfrak{o}_i\subset\mathfrak{o}_j$ for $i<j$ (assuming that the index
set is suitably ordered). The idea is that, for each pair $(a,b)$, there is some $\mathfrak{o}_i$ where $a,b$ factor
nicely enough (although possibly not into irreducibles) such that $\gcd(a,b), c, d$ exist in $\mathfrak{o}_i$ and such
that $\gcd(a,b), c, d$ keep the properties needed for lcm splitting in extension $\mathfrak{o}_j$ all the way up to
$\bar{\mathbb{Z}}$.

If one wanted to generalize this tower strategy to general GCD domains $R$, one would have to show first that $R$ admits
suitable towers. I do not see how this could be possible without using anything about the specifics of the ring in
question. It might also be necessary to make the tower depend on the given $a,b$. Note that Hutchins~\cite[Ex. 31,
  p. 66]{Hut81} mentions, for an example ring having similar properties as $R_1$, that any countable tower of PIDs is
B\'{e}zout. One would need some kind of converse theorem that any B\'{e}zout domain, potentially given some not to
restrictive conditions, can be expressed as a suitable tower of nice domains, but I am not aware of any result in that
direction. Anyway, that tower strategy might be worth looking at in further investigations of lcm splittings.

\end{enumerate}

\end{document}